\documentclass[leqno,10pt]{amsart}
\usepackage{amssymb}
\usepackage{dsfont}
\usepackage{color}

\newtheorem{theorem}{Theorem}[section]
\newtheorem{proposition}[theorem]{Proposition}
\newtheorem{lemma}[theorem]{Lemma}
\newtheorem{corollary}[theorem]{Corollary}
\theoremstyle{definition}
\newtheorem{definition}[theorem]{Definition}
\newtheorem{problem}[theorem]{Problem}
\newtheorem{example}[theorem]{Example}
\theoremstyle{remark}
\newtheorem{remark}[theorem]{Remark}

\numberwithin{equation}{section}

\DeclareMathOperator{\lip}{\text{Lip}}
\DeclareMathOperator{\erfc}{\text{erfc}}
\DeclareMathOperator{\Div}{div}

\begin{document}

\title[HJ equations with time-fractional derivative]{Well-posedness of Hamilton-Jacobi equations with Caputo's time-fractional derivative}

\author{Yoshikazu~Giga}
\address{Graduate School of Mathematical Sciences, The University of Tokyo 3-8-1 Komaba, Meguro-ku, Tokyo, 153-8914 Japan}
\email{labgiga@ms.u-tokyo.ac.jp}

\author{Tokinaga~Namba}
\address{Graduate School of Mathematical Sciences, The University of Tokyo 3-8-1 Komaba, Meguro-ku, Tokyo, 153-8914 Japan}
\email{namba@ms.u-tokyo.ac.jp}



\begin{abstract}
A Hamilton-Jacobi equation with Caputo's time-fractional derivative of order less than one is considered.
The notion of a viscosity solution is introduced to prove unique existence of a solution to the initial value problem under periodic boundary conditions.
For this purpose comparison principle as well as Perron's method is established.
Stability with respect to the order of derivative as well as the standard one is studied.
Regularity of a solution is also discussed.
Our results in particular apply to a linear transport equation with time-fractional derivatives with variable coefficients.
\end{abstract}

\maketitle

\section{Introduction}
Let $\alpha\in(0,1]$ and $0<T<\infty$ be given constants.
We consider the initial-value problem for the Hamilton-Jacobi equation of the form
\begin{equation}
\label{e:fhj}
\partial_{t}^{\alpha}u+H(t,x,u,Du)=0\quad\text{in $(0,T]\times\mathbb{T}^{d}=:Q_T$}
\end{equation}
and
\begin{equation}
\label{e:init}
u|_{t=0}=u_0\quad\text{in $\mathbb{T}^{d}$}.
\end{equation}
Here $\mathbb{T}^{d}:=\mathbb{R}^{d}/\mathbb{Z}^{d}$ is {the} $d$-dimensional torus, $u:\overline{Q_T}\to\mathbb{R}$ is an unknown function and $H:\overline{Q_T}\times\mathbb{R}\times\mathbb{R}^{d}\to\mathbb{R}$ is a given function called a \emph{Hamiltonian}.
Moreover, $Du$ denotes the spatial gradient, i.e., $Du=(\partial u/\partial x_1,\cdots,\partial u/\partial x_d)$ and
$\partial_{t}^{\alpha}u$ denotes \emph{Caputo's (time-)fractional derivative} which is defined by
\begin{equation}
\label{e:caputo}
(\partial_{t}^{\alpha}f)(t):=\begin{cases}
	\displaystyle\frac{1}{\Gamma(1-\alpha)}\int_{0}^{t}\frac{f'(s)}{(t-s)^{\alpha}}ds\quad&\text{for $\alpha\in(0,1)$,}\\
	f'(t)\quad&\text{for $\alpha=1$,}
	\end{cases}
\end{equation}
where $\Gamma(\cdot)$ is the usual gamma function.
Throughout this paper, a function $v$ on $\mathbb{T}^{d}$ is regarded as a function defined on $\mathbb{R}^{d}$ with $\mathbb{Z}^{d}$-periodically, i.e., $v(x+z)=v(x)$ for all $x\in\mathbb{R}^{d}$ and $z\in\mathbb{Z}^{d}$.
Although some part of our arguments can be easily extended to other boundary conditions we now restrict ourselves only under periodic boundary conditions.

{The goal of this paper is to find a proper notion of viscosity solutions so that \eqref{e:fhj}-\eqref{e:init} is well-posed. More specifically, we establish unique existence, stability and some regularity results of a viscosity solution for \eqref{e:fhj}-\eqref{e:init}}.
Here we will consider only for $\alpha\in(0,1)$ since the case of $\alpha=1$ has been well studied.
All results excepts for Section \ref{s:regularity} and Section \ref{s:comment} will be established under the following assumptions:
\begin{itemize}
\item[(A1)] $H:\overline{Q_T}\times\mathbb{R}\times\mathbb{R}^{d}\to\mathbb{R}$ is continuous,
\item[(A2)] there is a modulus $\omega:[0,\infty)\to[0,\infty)$ such that 
$$
|H(t,x,r,p)-H(t,y,r,p)|\le\omega(|x-y|(1+|p|))
$$
for all $(t,x,y,r,p)\in[0,T]\times\mathbb{R}^{d}\times\mathbb{R}^d\times\mathbb{R}\times\mathbb{R}^{d}$,
\item[(A3)] $r\mapsto H(t,x,r,p)$ is nondecreasing for all $(t,x,p)\in[0,T]\times\mathbb{T}^d\times\mathbb{R}^{d}$,
\item[(A4)] $u_0:\mathbb{T}^{d}\to\mathbb{R}$ is continuous.
\end{itemize}
We emphasize that these assumptions are fairly standard for $\alpha=1$.
Of course there might be several generalizations but we do not touch them.
Note that we do not assume coercivity, i.e.,
\begin{equation}
\label{e:coercive}
\liminf_{r\to\infty}\{H(t,x,r,p) \mid (t,x,r)\in Q_T\times\mathbb{R},|p|\ge r\}=+\infty.
\end{equation}
Hence our results apply to a transport equation
\begin{equation}
\label{e:transport}
\partial_t^{\alpha}u+b\cdot Du=0
\end{equation}
for $b=b(t,x):\overline{Q_T}\to\mathbb{R}^d$.

Since a notion of viscosity solutions was introduced by Crandall and Lions \cite{cl}, its theory has developed rapidly and by now there is a large number of literature.
The reader is referred to \cite{abig}, \cite{bcd} and \cite{k} for basic theory
and to \cite{cil} and \cite{g} for more advanced theory.
The theory of viscosity solutions had been applied initially to local partial differential equations (pdes for short) and soon has been extended by Soner \cite{s} to pdes with space-fractional derivatives which are defined non-locally.
See also \cite{bi}, \cite{ai}, {\cite{cs}, \cite{i}} and references therein.
In these papers the authors are commonly interested in L\'{e}vy operators, which can be represented (formally) as
\begin{equation}
\label{e:levy}
g[f](x)=-\int_{\mathbb{R}^d}\left(f(x+z)-f(x)-\frac{Df(x)\cdot z}{1+|z|^2}\right)d\mu(z),
\end{equation}
where $d\mu$ is the L\'{e}vy measure.
An example of L\'{e}vy operators is the fractional Laplacian:
\begin{equation}
\label{e:flaplacian}
(-\Delta)^{\alpha}f(x)=C\int_{\mathbb{R}^d}\frac{f(x)-f(y)}{|x-y|^{d+2\alpha}}dy,
\end{equation}
where $C$ is a constant depending on $d$ and $\alpha$.

Above works are motivated from applied fields such as physics, engineering and finances. 
Applicabilities of pdes with time-fractional derivatives has been discussed by many researchers in wide fields as well; \cite{fch}, \cite{fgd}, \cite{smb} and \cite{v} for instance.
We here refer to several mathematical works for pdes with Caputo's time-fractional derivative (CTFD for short) in order to motivate our research.
Although many definitions of a different kind of fractional derivatives have been suggested, we will not touch them in this paper and instead the reader is referred to \cite{d}, \cite{gkmr}, \cite{krs}, \cite{kst}, \cite{p}, \cite{skm} and \cite{z}.
A typical example of pdes with CTFD is
\begin{equation}
\label{e:anomolous}
\partial_t^{\alpha}u+L(u)=F,
\end{equation}
where $L$ consists of a symmetric uniformly elliptic operator and a transport term and $F=F(t,x)$ is a given function.
This can be considered as {an} equation describing diffusion phenomena in complex media like fractals and then is called \emph{anomalous diffusion} or \emph{singular diffusion}.
There seem to be several previous works for \eqref{e:anomolous} (see \cite{sy} and references therein) and Luchko's works have a close relationship with ours.
He established a maximum principle for Caputo's fractional derivative in \cite{l} and, based on it, proved a uniqueness of classical solutions for an initial-boundary value problem of \eqref{e:anomolous} with the type of $L(u)=-\Div(p(x)Du)+q(x)u$, where typically $p$ is smooth and uniformly positive with continuous $q$.  
In \cite{l2} he established an existence of classical solutions for same equations {and gave some existence results of a continuous generalized solution as well}.
His research has been continued in a work by Sakamoto and Yamamoto \cite{sy}, which is a pioneer work in the theory of weak solutions for \eqref{e:anomolous}.
They defined {a weak solution} in the sense of distribution for a similar equation as one Luchko considered and established well-posedness in order to consider inverse problems.
Researches on this line have been growing rapidly; see, e.g., \cite{lly} for multi-term time-fractional derivatives and \cite{lryz} for {\eqref{e:anomolous}} with {semilinear} source terms.

{Several pdes with strong nonlinearity have also been considered.} Anomalous diffusion equations are modeled by the continuous-time random walk (CTRW for short) introduced by Montroll and Weiss (\cite{mw}).
Kolokoltsov and Veretennikova (\cite{kv}) extended the notion of CTRW so that its processes can be controlled and then derived (heuristically) Hamilton-Jacobi-Bellman equations with CTFD, fractional Laplacian and some additional term.
We note that this does not include second order spatial-derivative.
In \cite{kv2} they also defined mild solutions that belong to $C^{1}([0,T];C^{1}_{\infty}(\mathbb{R}^d))$ and proved well-posedness for an initial-value problem of
$$
\partial_t^{\alpha}u=-a(-\Delta)^{\beta/2}u+H(t,x,Du).
$$
Here $C^{1}_{\infty}(\mathbb{R}^d)$ is a set of $C^1$ functions that decreasing rapidly at infinity and $\beta\in(1,2]$, $a>0$ are given constants {and $H$ is a given Lipschitz continuous function}.
{In \cite{a} Allen extended the notion of viscosity solutions to the time-space nonlocal equation with CTFD of the form
$$
\partial_t^{\alpha}u-\sup_i\inf_j\left(\int_{\mathbb{R}^d}\frac{u(t,x+y)-u(t,x)}{|y|^{d+2\sigma}}a^{ij}(t,x,y)dy\right)=f
$$
and discussed regularity problems of the solutions.
Here $a_{ij}$ is positive, bounded function that is symmetric with respect to the third variable and $f$ is a given function.
To the best of our knowledge, this seems to be the only results for pdes with CTFD by a viscosity approach. He defined the solutions based on the idea of viscosity solutions by treating 
\begin{align}
K_0[f](t)&=\frac{f(t)-f(0)}{t^\alpha\Gamma(1-\alpha)}+\frac{\alpha}{\Gamma(1-\alpha)}\int_0^t\frac{f(t)-f(s)}{(t-s)^{\alpha+1}}ds\label{e:k01}\\
&=\frac{\alpha}{\Gamma(1-\alpha)}\int_{-\infty}^t\frac{\tilde{f}(t)-\tilde{f}(s)}{(t-s)^{\alpha+1}}ds;\label{e:k02}
\end{align}
in \cite{a} this quantity is denoted by $\partial_t^\alpha$.
Here $\tilde{f}:(-\infty,T]\to\mathbb{R}$ is an extension of $f:[0,T]\to\mathbb{R}$ by $\tilde{f}(t)=f(0)$ for $t<0$. Note that $\partial_t^\alpha f=K_0[f]$ by integrating by parts if $f$ is smooth; see Proposition \ref{p:integrationbyparts}. He dealt with CTFD by $K_0$ in the similar way as space-fractional derivatives mentioned above since the form of integration in $K_0$ is very close to fractional Laplacian. However, the well-posedness remained to be unclear (at least for non-experts). The operator $K_0$ is strictly different from the fractional Laplace operator since there are several different properties between their kernels such as symmetry. In addition, the fractional Laplacian (or the L\'{e}vy operator) is suitable for boundary value problems, whereas Caputo's derivative is suitable for initial value problems. From these facts, studying CTFD is not a simple adjustment of the space-fractional cases.}

We explained so far second-order pdes or first-order pdes including fractional Laplacian with CTFD.
For first-order pdes with CTFD but without any higher-order terms than one such as \eqref{e:transport}, a formula of solution for \eqref{e:transport} is given by Mainardi, Mura and Pagnini (\cite{mmp}) for instance, but for equations with constant coefficients. 
However, there seems to be no {general frameworks}.
{In light of the above situations, we aim to construct the synthetic theory of viscosity solutions so that (fully nonlinear) pdes with CTFD mentioned above can be considered.}
However, extensions to second order equations expects some technical issues, so we only treat first order equations in this paper as the first step.
For second order problems the reader is referred to one of forthcoming papers of the second author.

We motivate our definition of viscosity solutions (Definition \ref{d:vsol}) by recalling the case of $\alpha=1$.
Let us suppose that $u$ is a classical subsolution of \eqref{e:fhj}, that is,
$$
(\partial_t^\alpha u)(t,x)+H(t,x,u(t,x),Du(t,x))\le0
$$
for all $(t,x)\in Q_T$ 
and that
\begin{equation}
\label{e:basicassumption}
\max_{[0,T]\times\mathbb{R}^d}(u-\phi)=(u-\phi)(\hat{t},\hat{x})
\end{equation}
for a {test} function $\phi$.
The classical maximum principle in space implies that $Du=D\phi$ at $(\hat{t},\hat{x})$.
With respect to time, the maximum principle for CTFD (\cite[Theorem 1]{l}) implies that $\partial_{t}^{\alpha}u\ge\partial_{t}^{\alpha}\phi$ at $(\hat{t},\hat{x})$.
Hence, that $u$ is a classical subsolution yields to
\begin{equation}
\label{e:nvsol}
(\partial_{t}^{\alpha}\phi)(\hat{t},\hat{x})+H(\hat{t},\hat{x},u(\hat{t},\hat{x}),D\phi(\hat{t},\hat{x}))\le0.
\end{equation}
In the spirit of the case of $\alpha=1$, it is natural to define {a weak subsolution} of \eqref{e:fhj} by {means of} \eqref{e:nvsol}.
Let us call it {a \emph{provisional subsolution}} in this paper.
Similarly, if $u$ is a classical supersolution and we replace the maximum by a minimum in \eqref{e:basicassumption}, then the converse inequality of \eqref{e:nvsol} is led.
Then let us call {$u$ defined by means of the inequality a \emph{provisional supersolution}} of \eqref{e:fhj}.
Let us call $u$ a \emph{provisional solution} of \eqref{e:fhj} if it is a both provisional sub- and supersolution of \eqref{e:fhj}.

{The notion of} provisional solutions looks easy to deal with but it is technically difficult to establish a comparison principle, so we do not know whether it is a proper notion of solution {for \eqref{e:fhj}} or not (see Section \ref{s:comment} for {some observations}).
One reason is that the so-called \emph{doubling variable method} (see, e.g., \cite[Section 3.3]{g}) does not work and a main problem is, roughly speaking, that $(\partial_t^{\alpha}\phi)(\hat{t},\hat{x})$ in \eqref{e:nvsol} is not an appropriate substitute of $(\partial_t^{\alpha}u)(\hat{t},\hat{x})$.
In a proof of comparison principle in the theory of viscosity solutions, we often aim to derive a contradiction by using the doubling variable method under a suitable supposition.
For provisional sub/supersolutions, we cannot derive a contradiction because of unnecessary values caused by $\partial_t^{\alpha}\phi$.

This fact makes us realize that it is necessary to {find} a function that has a closer value to $\partial_{t}^{\alpha}u$ at each point. 
{It becomes now clear that it is better to handle the operator $K_0$ considered by Allen than $\partial_t^\alpha$. 
Strictly speaking, Allen adopted \eqref{e:k02}, but we adopt 
$$
K_0[u](t,x)=\frac{u(t,x)-u(0,x)}{t^\alpha \Gamma(1-\alpha)}+\frac{\alpha}{\Gamma(1-\alpha)}\int_0^t(u(t,x)-u(t-\tau,x))\frac{d\tau}{\tau^{\alpha+1}}
$$
for technical reasons, which is derived from \eqref{e:k01} by changing the variable of integration.}
Here the integral is interpreted as an improper integral
\begin{equation}
\label{e:improperintegral}
\int_{0}^{t}(u(t,x)-u(t-\tau,x))\frac{d\tau}{\tau^{\alpha+1}}=\lim_{r\searrow0}\int_{r}^{t}(u(t,x)-u(t-\tau,x))\frac{d\tau}{\tau^{\alpha+1}}.
\end{equation}
Since the convergence of (improper) integration \eqref{e:improperintegral} is not trivial, we apply test functions instead of unknown functions near the singular time, i.e., the lower end of interval {of integration}.
This idea is taken in our definition.
{We note that our solution is the same as Allen's except for handling non-locality in the spatial direction since this idea is similar as for him (or space-fractional derivatives).}

Let us give strategies of our proofs in this paper.
{They are basically similar to space-fractional cases with integer-order time derivatives (\cite{bi}, \cite{ai}, \cite{cs} and \cite{i}) although there are some differences depending on the role of time and space derivatives.}
We will show that \eqref{e:improperintegral} exists as a finite number at points such that $u-\phi$ attains a maximum/minimum (Lemma \ref{l:equivalence}), where $\phi$ is a test function.
This is an analogy of \cite{bi} or \cite[Lemma 3.3]{cs} and {the proof follows similar arguments}.
{This fact enables us to prove the comparison principle.}
An idea of the proof of comparison principle is the same as usual ones, that is, doubling variable method under contradiction.
{We note that the space-fractional case has an elliptic flavor while CTFD has an evolutional flavor, so the details of the proof are different from the space-fractional case.
Indeed, in the process of doubling variables we subtract the inequality for a viscosity supersolution from the inequality for a viscosity subsolution.
For space-fractional cases, we reach a contradiction from time derivative of test functions while all spatial quantities are canceled.
On the other hand, in our case the contradiction follows from the term $K_0$, so we have to keep this term until the end.
More precisely, the contradiction is derived from terms having no integration in a difference of $K_0$ for viscosity sub- and supersolutions.
Note also that test functions like $\eta/(T-t)$ or $\eta t$ seem to be not helpful and even unnecessary in an auxiliary function, where $\eta$ is a positive constant.
}

A proof of existence of viscosity solutions follows Perron's method, that is, a construction of maximal subsolutions.
Since \eqref{e:fhj} is nonlocal, we need some efforts to handle nonlocal terms compared with the case of local equations.
For this purpose we will employ the idea used in \cite[Theorem 3]{i} for example.

We will establish stability results under limit operation from two perspectives.
One of them is for a family of solutions of \eqref{e:fhj} with a Hamiltonian depending on a parameter.
Our statement and proof are almost the same as for $\alpha=1$ (see, e.g., \cite{bcd}). 
The other stability discussed here is the case when time-derivative's orders are regarded as parameters.
The latter can be proved under the same idea as the former by defining analogous functions of half-relaxed limits. 

We will show that viscosity solutions are Lipschitz continuous in space and $\alpha$-H\"{o}lder continuous in time under some additional assumptions on $H$ and $u_0$.
When the regularity problems for viscosity solutions are discussed, the coercivity condition is often assumed.
However, transport equations are not coercive.
In view of applications we will derive the above regularity results without the coercivity assumption.
Our proofs follow basically ones for $\alpha=1$ (\cite{abig}) but a proof of the temporal regularity may be not standard.
We will construct a viscosity solution of \eqref{e:fhj}-\eqref{e:init} that is $\alpha$-H\"{o}lder continuous in time by Perron's method for a family of viscosity subsolutions of \eqref{e:fhj}-\eqref{e:init} that is $\alpha$-H\"{o}lder continuous in time.
In this argument we should be carefully for a dependence of the H\"{o}lder constant of viscosity subsolutions.
We will show that viscosity solutions are H\"{o}lder continuous in time at the initial time, where its H\"{o}lder constant depends only on $T$, $\alpha$, $H$ and $u_0$.
By restricting viscosity subsolutions to H\"{o}lder continuous functions with such a constant, we will obtain viscosity solution with the desired regularity.

Our results are new even for transport equations \eqref{e:transport} with variable coefficients although a formula of solution for constant coefficients is known only in the one dimensional case (see Section \ref{s:regularity}). 
For this reason it is worth summarizing here.

\begin{theorem}
Let $b:\overline{Q_T}\to\mathbb{R}^d$ be a continuous function.
Assume that there is constants $C_1>0$ and $C_2>0$ such that 
\begin{equation}
\label{e:coefficient}
|b(t,x)-b(t,y)|\le C_1|x-y|
\end{equation}
and
\begin{equation}
\label{e:lipschitzinitial}
|u_0(x)-u_0(y)|\le C_2|x-y|
\end{equation}
for all $(t,x,y)\in(0,T]\times\mathbb{R}^d\times\mathbb{R}^d$.
Then there exists at most one viscosity solution $u\in C(\overline{Q_T})$ of \eqref{e:transport}-\eqref{e:init}.
Moreover there exists a constant $C_3>0$ such that
\begin{equation}
\label{e:regularity}
|u(t,x)-u(s,y)|\le C_3(|t-s|^{\alpha}+|x-y|)
\end{equation}
for all $(t,x,s,y)\in([0,T]\times\mathbb{R}^d)^2$.
\end{theorem}
If one does not require \eqref{e:regularity}, conditions \eqref{e:coefficient} and \eqref{e:lipschitzinitial} can be weakened so that Hamiltonian $H(t,x,p)=b(t,x)\cdot p$ satisfies (A1) and (A2).

We finally compare with viscosity solutions with weak solutions in the sense of distribution (weak solution for short) and mention several open problems.
A weak solution for linear second order problems \eqref{e:anomolous} given by Sakamoto and Yamamoto (\cite{sy}) was constructed by {a} Galerkin method.
Since approximate equations have no comparison principle, it is difficult to compare two notions under the current circumstances.
Of course, the case of $\alpha=1$ has the same difficulty and, even such simple looking case, there seems to be few literatures (\cite{hl}, \cite{ishii3}, \cite{jlm} and \cite{jlp}).
In order to overcome such a difficulty, analyses for further regularities of weak solutions in the both senses will be needed.

As another direction of researches for weak solutions of pdes with CTFD, we should mention fractional derivative of the form 
\begin{equation}
\label{e:caputooriginal}
(D^\alpha_tf)(t):=\frac{1}{\Gamma(1-\alpha)}\frac{d}{dt}\int_0^t\frac{f(\tau)-f(0)}{(t-\tau)^{\alpha+1}}d\tau.
\end{equation}
The original definition of Caputo's fractional derivative by himself was actually given as this form and hence this derivative is also called Caputo's fractional derivative.
We note that $D^\alpha_t f=\partial_t^\alpha f$ almost everywhere on $[0,T]$ if $f$ is absolutely continuous on $[0,T]$. 
See \cite[Chapter 3]{d} for a brief history of Caputo's fractional derivative and the above relationship between two definitions.
There are some works for weak solutions of pdes with \eqref{e:caputooriginal}.
Zacher (\cite{za}) considered abstract evolutional equations of parabolic type including 
$$
D^\alpha_tu-\Div(ADu)+b\cdot Du+cu=0
$$
and, by introducing a notion of a weak solution, he established a unique existence.
Here, $A=A(t,x)$ is a symmetric and positive defined matrix-valued function with $L^\infty$ elements and $b=b(t,x)$ and $c=c(t,x)$ are $L^\infty$ functions.
See \cite{za2}, \cite{za3} and \cite{krr} for related works.
An analysis of weak solutions for pdes with \eqref{e:caputooriginal} involves the problem of the trace $u(0,\cdot)$ of $u$ up to $t=0$ since $D_t^\alpha u$ includes the value $u(0,\cdot)$.
This needs some regularity up to $t=0$ which forced as to restrict range of $\alpha$, say, for example $\alpha>1/2$ or regularity of some of given functions $A$, $b$ and $c$ compared with the case $\alpha=1$ (\cite{ku}).
We note that such a trace problem was not considered in \cite{za}; moreover, assumptions of \cite{za} seem to be too weak to get necessary regularity.
In view of such restrictions, our viscosity solutions might look a better notion of weak solutions since we are able to obtain a continuous (viscosity) solution for every $\alpha\in(0,1)$ with no special assumptions on $H$.
However, we cannot compare two notions since it is not guaranteed that our solution $u$ is absolutely continuous in time, so it is not clear whether or not $D_t^\alpha u=\partial_t^\alpha u$ for our solution.
Even for this problem, further analyses from both aspects of viscosity solutions and weak solutions are needed. 

This paper is organized as follows: In Section 2 we give a definition of viscosity solutions after and summarize some facts used in the other sections. In Section 3 we prove a comparison principle and in Section 4 we establish an existence result.
In Section 5 we prove two types of stability results and in Section 6 we study regularity problem for \eqref{e:fhj}. Finally, in Section 7 we give a definition of provisional solutions as another possible notion of weak solutions and mention the technical difficulty for them.

\section{Definition and properties of solutions}\label{s:defsol}
In this section we assume that Hamiltonian $H$ is merely continuous on $\overline{Q_T}\times\mathbb{R}\times\mathbb{R}^{d}$.

\subsection{Preliminaries}
To give a definition of viscosity solutions we first introduce a function space of the type
$$
\mathcal{C}^{1}([a,b]\times O):=\{\phi\in C^{1}((a,b]\times O)\cap C([a,b]\times O) \mid \text{$\partial_{t}\phi(\cdot,x)\in L^{1}(a,b)$ for every $x\in O$}\}.
$$
Here $a,b\in\mathbb{R}$ are constants such that $a<b$, $O$ is a domain in $\mathbb{R}^d$, $\mathbb{T}^d$ and $\mathbb{R}^{d}\times\mathbb{R}^d$ and $L^1(a,b)$ is the space of Lebesgue integrable functions on $(a,b)$. 
Note that $u\in \mathcal{C}^{1}([a,b]\times O)$ may not be $C^1$ up to $t=a$.
This space will be used as a space of test functions as well as of classical solutions of \eqref{e:fhj}-\eqref{e:init}.
Here we define classical solutions of \eqref{e:fhj}-\eqref{e:init} as follows: 

\begin{definition}[Classical solutions]
A function $u\in\mathcal{C}^1(\overline{Q_T})$ is called a \emph{classical solution} of \eqref{e:fhj}-\eqref{e:init} if $u(0,\cdot)=u_0$ on $\mathbb{T}^d$ and
$$
(\partial_t^{\alpha}u)(t,x)+H(t,x,u(t,x),Du(t,x))=0
$$
for all $(t,x)\in Q_T$.
\end{definition}

Note that $\partial_t^{\alpha}\phi$ is bounded in $(0,T]\times\mathbb{R}^d$ if $\phi\in\mathcal{C}^1([0,T]\times\mathbb{R}^d)$; see \cite[Theorem 2.1]{d}.
We are tempted to use $C^{1}([a,b]\times O)$ as a space of classical solutions since the integrability condition for $\partial_t\phi(\cdot,x)$ is satisfied if $\phi$ belongs to it: $C^{1}([a,b]\times O)\subset\mathcal{C}^{1}([a,b]\times O)$.
However, the class $C^{1}([a,b]\times O)$ is too narrow to define classical solutions since it is necessary to include functions that have a fractional power with respect to time at the initial time such as $t^{\alpha}$.
That is why we do not assume the differentiability at the initial time.

\begin{example}
As an example let us consider a simple ordinary differential equation of the form
$$
\partial_{t}^{\alpha}f+f=0\quad\text{in $(0,\infty)$}
$$
with prescribed data $f(0)=c\in\mathbb{R}$.
According to \cite[Theorem 4.3]{d} a solution of this equation is given as $f(t)=c E_{\alpha}(-t^{\alpha})$, where $E_{\alpha}$ is the Mittag-Leffler function defined by
$$
E_{\alpha}(z):=\sum_{j=0}^{\infty}\frac{z^j}{\Gamma(j\alpha+1)}.
$$
In particular, $E_{1/2}(-\sqrt{t})=e^t \erfc(\sqrt{t})$, where $\erfc$ is the complementary error function defined by
$$
\erfc(z):=\frac{2}{\sqrt{\pi}}\int_{z}^{\infty}e^{-t^2}dt.
$$
The function $f$ is not differentiable at $t=0$ though it is continuous up to $t=0$; we leave the verification to the reader.
Therefore classical solutions of equations with Caputo's (time-)fractional derivative are not always differentiable at the initial time even if an initial datum is smooth.
\end{example}

For a measurable function $f:[0,T]\to\mathbb{R}$
we define functions $J_r[f],K_r[f]:(0,T]\to\mathbb{R}$ by
$$
J_r[f](t):=\frac{\alpha}{\Gamma(1-\alpha)}\int_0^r(f(t)-f(t-\tau))\frac{d\tau}{\tau^{\alpha+1}}
$$
and
$$
K_r[f](t):=\frac{f(t)-f(0)}{t^{\alpha}\Gamma(1-\alpha)}+\frac{\alpha}{\Gamma(1-\alpha)}\int_r^t(f(t)-f(t-\tau))\frac{d\tau}{\tau^{\alpha+1}}
$$
with a parameter $r\in(0,t)$.
For a measurable function $f:[0,T]\times\mathbb{R}^\ell\to\mathbb{R}$ with $\ell\ge1$ we define $J_r[f],K_r[f]:(0,T]\times\mathbb{R}^\ell\to\mathbb{R}$ by $J_r[f](t,x):=J_r[f(\cdot,x)](t)$ and $K_r[f](t,x):=K_r[f(\cdot,x)](t)$ for $(t,x)\in(0,T]\times\mathbb{R}^\ell$.

\begin{proposition}[Integration by parts]
\label{p:integrationbyparts}
Let $f:[a,T]\to\mathbb{R}$ be a function such that $f\in C^{1}((a,T])\cap C([a,T])$ and $f'\in L^{1}(a,T)$, where $a<T$.
Then
$$
\frac{1}{\Gamma(1-\alpha)}\int_a^t\frac{f'(\tau)}{(t-\tau)^{\alpha}}d\tau
=\frac{f(t)-f(a)}{(t-a)^{\alpha}\Gamma(1-\alpha)}+\frac{\alpha}{\Gamma(1-\alpha)}\int_0^{t-a}(f(t)-f(t-\tau))\frac{d\tau}{\tau^{\alpha+1}}.
$$
\end{proposition}

\begin{proof}
The left-hand side (multiplied by $\Gamma(1-\alpha)$) can be calculated as
\begin{align*}
\int_a^t\frac{f'(\tau)}{(t-\tau)^{\alpha}}d\tau
&=\int_a^t\frac{\frac{d}{d\tau}(f(\tau)-f(t))}{(t-\tau)^{\alpha}}d\tau\\
&=\left[\frac{f(\tau)-f(t)}{(t-\tau)^{\alpha}}\right]_a^t-\alpha\int_a^t\frac{f(\tau)-f(t)}{(t-\tau)^{\alpha+1}}d\tau\\
&=\lim_{\tau\to t}\frac{f(\tau)-f(t)}{(t-\tau)^{\alpha}}+\frac{f(t)-f(a)}{(t-a)^{\alpha}}+\alpha\int_a^t\frac{f(t)-f(\tau)}{(t-\tau)^{\alpha+1}}d\tau.
\end{align*}
Thanks to the smoothness of $f$, the first term vanishes.
By the change of variable $s:=t-\tau$ we obtain the desired result. 
\end{proof}

Let us share some words for an integral
$$
I[f](t):=\int_{a}^{b}f(t,\tau)\frac{d\tau}{\tau^{\alpha+1}}
$$
for constants $a,b\in\mathbb{R}$ with $0\le a<b\le T$ and a measurable function $f:[0,T]\times[0,T]\to\mathbb{R}$.
We say that the integral $I[f]$ \emph{makes sense} if either $I[f^+]$ or $I[f^-]$ is finite (in the sense of Lebesgue integrals) and that $I[f]$ \emph{exists} if both $I[f^{\pm}]$ are finite.
Here $f^{\pm}:=\max\{\pm f,0\}$.
It is necessary to pay attention when $a=0$.
Then we regard $I[f]$ as an improper integral by $I[f](t)=\lim_{r\searrow0}I_r[f]$, where
$$
I_r[f](t)=\int_r^bf(t,\tau)\frac{d\tau}{\tau^{\alpha+1}}.
$$
Thus $I[f]$ exists if $I_r[f^{\pm}]$ are finite for each $r$ and $\lim_{r\searrow0}I_r[f^{\pm}]$ exist as a finite number.
Note that, if $\tau\mapsto |f(t,\tau)|/\tau^{\alpha+1}$ is integrable on $(0,b)$, then $I[f]$ exits and it agrees with the Lebesgue integral; this is a direct consequence of the dominated convergence theorem.  
We abuse above words not only for $J_r$ but also for $K_r$ including a non-integration term.

For a set $E\subset\mathbb{R}^{\ell}$ with $\ell\ge1$, let $USC(E)$ and $LSC(E)$ be sets of real-valued upper and lower semicontinuous functions on $E$, respectively. 
Note that semicontinuous functions are measurable.

\begin{proposition}[Properties of $J_r$ and $K_r$]
\label{p:propertyintegral}
Let $f\in USC([0,T])$ (resp. $LSC([0,T])$) and $g\in C^1((0,T])$. 
Then 
\begin{itemize}
\item[(i)] for each $t\in(0,T]$, $J_r[g](t)$ exists for all $r\in(0,t)$,
\item[(ii)] for each $t\in(0,T]$, $K_r[f](t)$ makes sense and is bounded from below (resp. above) for all $r\in(0,t)$, 
\item[(iii)] $K_0[f](\hat{t})$ makes sense and is bounded from below (resp. above) if $f-g$ attains a maximum (resp. minimum) at $\hat{t}\in (0,T]$ over $(0,T]$, i.e.,
$$
\sup_{(0,T]}(f-g)=(f-g)(\hat{t})\quad(\text{resp. $\inf_{(0,T]}(f-g)=(f-g)(\hat{t})$}),
$$
\end{itemize}
Moreover for each $j\ge0$ let $t_j\in(0,T]$, $r_j\in(0,t_j)$ and $\alpha_j\in(0,1)$ be sequences such that $\lim_{j\to\infty}(t_j,r_j,\alpha_j)=(\hat{t},\hat{r},\alpha)\in(0,T]\times\mathbb{R}^d\times[0,\hat{t})\times(0,1)$.
Let $J_r^{\alpha_j}$ denote a function $J_r$ associated with $\alpha=\alpha_j$.
Then
\begin{itemize}
\item[(iv)] $\lim_{j\to\infty}J_{r_j}^{\alpha_j}[g](t_j)=J_{\hat{r}}^{\alpha}[g](\hat{t})$.
\end{itemize}
\end{proposition}

\begin{proof}
(i) Fix $t\in(0,T]$ and $r\in(0,t)$ arbitrarily.
Since $g$ is Lipschitz continuous near $t$ due to the smoothness of $g$, for some constant $C>0$
$$
\int_0^r|g(t)-g(t-\tau)|\frac{d\tau}{\tau^{\alpha+1}}\le \int_0^rC\tau\frac{d\tau}{\tau^{\alpha+1}}=\frac{Cr^{1-\alpha}}{1-\alpha}.
$$
Our assertion follows immediately from this.

(ii) Fix $t\in(0,T]$ and $r\in(0,t)$ arbitrarily.
Assume that $f\in USC([0,T])$.
Then $f$ attains a maximum and hence
$$
f(t)-\max_{[0,T]}f\le f(t)-f(t-\tau)
$$
for all $\tau\in(r,t)$.
The left-hand side multiplied by $\tau^{-\alpha-1}$ is integrable on $(r,t)$ since we integrate away from $\tau=0$.
Therefore the negative part $[f(t)-f(t-\tau)]^{-}/\tau^{\alpha+1}$ is integrable on $(r,t)$.
This implies that $K_r[f](t)$ makes sense and is bounded from below.
The similar argument is applied for $f\in LSC([0,T])$.
The above yields our assertion.

(iii) Define $h:=g+(f-g)(\hat{t})$ and
\begin{equation}
	v(\tau):=\begin{cases}
		h(\hat{t})-h(\hat{t}-\tau)\quad&\text{for $\tau\in[0,\hat{t}/2]$},\\
		f(\hat{t})-f(\hat{t}-\tau)\quad&\text{for $\tau\in(\hat{t}/2,\hat{t}]$.}
	\end{cases}
\end{equation}
Since $f-h$ attains a maximum at $\hat{t}$ over $(0,T]$, we see $f\le h$ on $(0,T]$.
In addition, $(f-h)(\hat{t})=0$ and thus
$$
h(\hat{t})-h(\hat{t}-\tau)\le f(\hat{t})-f(\hat{t}-\tau)
$$
on $(0,\hat{t})$.
By (i) and a similar argument as the proof of (ii) with $r=\hat{t}/2$ it turns out that the negative part $v^-(\tau)/\tau^{\alpha+1}$ in integrable on $(0,\hat{t})$, so is $[f(\hat{t})-f(\hat{t}-\tau)]^-/\tau^{\alpha+1}$ since $v(\tau)\le f(\hat{t})-f(\hat{t}-\tau)$ on $(0,\hat{t})$.
This yields our assertion for $f\in USC([0,T])$.
Another can be proved similarly.

(iv) Thanks to the smoothness of $g$ the dominated convergence theorem can be applied and ensures our assertion.
More precisely, since $\inf_{j\ge0}(t_j-r_j)>0$ and $g\in C^1((0,T])$, there exists a constant $C_1>0$ such that $|g(t_j)-g(t_j-\tau)|\le C_1\tau$ on $(0,r_j)$.
In particular, we may assume that $C_1$ does not depend on $j$ since $\lim_{j\to\infty}t_j=\hat{t}>0$.
Thus we have
\begin{equation}
\label{e:estimatetestfunction}
\sup_{j\ge0}(|g(t_j)-g(t_j-\tau)|\mathds{1}_{(0,r_j)}(\tau))\tau^{-\alpha-1}\le C_1\mathds{1}_{(0,\hat{r})}(\tau)\tau^{-\alpha}
\end{equation}
for all $\tau\in[0,T]$.
Here $\mathds{1}_I$ is the indicator function on an interval $I$, i.e., $\mathds{1}_I=1$ in $I$ and $0$ elsewhere. 
The right-hand side is integrable on $[0,T]$.
It remains to check the convergence of $(g(t_j)-g(t_j-\cdot))\mathds{1}_{[0,r_j]}(\cdot)$ but this is obvious.
\end{proof}

\subsection{Definition of solutions}
We now give our definition of viscosity solutions for \eqref{e:fhj}.

\begin{definition}[Viscosity solutions]
\label{d:vsol}
A function $u\in USC(\overline{Q_T})$ (resp. $LSC(\overline{Q_T})$) is called a \emph{viscosity subsolution} of \eqref{e:fhj} if, for any constants $a,b\in[0,T]$ with $a<b$ and an open ball $B$ in $\mathbb{R}^d$,
\begin{equation}
\label{e:subineq}
J_{\hat{t}-a}[\phi](\hat{t},\hat{x})+K_{\hat{t}-a}[u](\hat{t},\hat{x})+H(\hat{t},\hat{x},u(\hat{t},\hat{x}),D\phi(\hat{t},\hat{x}))\le0\quad(\text{resp. $\ge0$})
\end{equation}
whenever $u-\phi$ attains a maximum (resp. minimum) at $(\hat{t},\hat{x})\in (a,b]\times B$ over $[a,b]\times\overline{B}$ for $\phi\in\mathcal{C}^1([0,T]\times\mathbb{R}^d)$.

If $u\in C(\overline{Q_T})$ is both a viscosity sub- and supersolution of \eqref{e:fhj}, then we call $u$ a \emph{viscosity solution} of \eqref{e:fhj}.
\end{definition}

\begin{remark}
(i) For an arbitrary function $u:Q_T\to\mathbb{R}$ an \emph{upper semicontinuous envelope} $u^*:\overline{Q_T}\to\mathbb{R}\cup\{\pm\infty\}$ and a \emph{lower semicontinuous envelope} $u_*:\overline{Q_T}\to\mathbb{R}\cup\{\pm\infty\}$ are defined by
$$
u^*(t,x):=\lim_{\delta\searrow0}\sup\{u(s,y) \mid (s,y)\in Q_T\cap\overline{B_{\delta}(t,x)}\}
$$
and $u_*:=-(-u)^*$.
Here $B_{\delta}(t,x)$ is an open ball of radius $\delta$ centered at $(t,x)$ and $\overline{B_{\delta}(t,x)}$ is its closure.
As for $\alpha=1$ a viscosity sub- and subsolution of \eqref{e:fhj} can be defined for arbitrary functions $u:Q_T\to\mathbb{R}$ by using $u^*$ (for a subsolution) and $u_*$ (for a supersolution) in Definition \ref{d:vsol}, where it is further assumed that $u^*<+\infty$ and $u_*>-\infty$ on $\overline{Q_T}$; cf. \cite[Definition 2.1.1]{g}.
Note that functions $u^*$ and $u_*$ are upper semicontinuous and lower semicontinuous on $\overline{Q_T}$, respectively (see, e.g., \cite[Proposition V.2.1]{bcd}) so they are measurable.

(ii) Although we restrict ourselves for spatially periodic functions, our definition can be easily extended for $(0,T]\times\Omega$, where $\Omega$ is a domain in $\mathbb{R}^d$.
In fact, the comparison principle holds for a general bounded domain with necessary modifications. 
\end{remark}

If a viscosity subsolution (resp. supersolution) $u$ of \eqref{e:fhj} satisfies $u(0,\cdot)\le u_0$ (resp. $u(0,\cdot)\ge u_0$) on $\mathbb{T}^{d}$, $u$ is called a viscosity subsolution (resp. viscosity supersolution) of \eqref{e:fhj}-\eqref{e:init}.
We often suppress the word ``viscosity" unless confusion occurs.

\subsection{Properties and equivalences of solutions}

\begin{proposition}[Replacement of test functions]
\label{p:replacement}
A function $u\in USC(\overline{Q_T})$ (resp. $LSC(\overline{Q_T})$) is a subsolution (resp. supersolution) of \eqref{e:fhj} if and only if, for any $a,b\in[0,T]$ with $a<b$ and an open ball $B$ in $\mathbb{R}^d$, \eqref{e:subineq} holds whenever
\begin{itemize}
\item[(i)] $u-\phi$ attains a zero maximum (resp. minimum) at $(\hat{t},\hat{x})\in (a,b]\times B$ over $[a,b]\times\overline{B}$ for $\phi\in\mathcal{C}^1([0,T]\times\mathbb{R}^d)$ such that $(u-\phi)(\hat{t},\hat{x})=0$, or
\item[(ii)] $u-\phi$ attains a strict maximum (resp. minimum) at $(\hat{t},\hat{x})\in(a,b]\times B$ over $[a,b]\times\overline{B}$ for $\phi\in\mathcal{C}^1([0,T]\times\mathbb{R}^d)$, i.e., 
\begin{align*}
&\max_{[a,b]\times\overline{B}}(u-\phi)=(u-\phi)(\hat{t},\hat{x})>(u-\phi)(t,x)\\(\text{resp. }&\min_{[a,b]\times\overline{B}}(u-\phi)=(u-\phi)(\hat{t},\hat{x})<(u-\phi)(t,x))
\end{align*}
for all $(t,x)\in[a,b]\times\overline{B}$.
\item[(iii)] $u-\phi$ attains a maximum (resp. minimum) at $(\hat{t},\hat{x})\in(a,b]\times B$ over $[a,b]\times\overline{B}$ for $\phi\in\mathcal{C}^1([a,b]\times\overline{B})$.
\end{itemize}
\end{proposition}

\begin{proof}
We only prove for a subsolution since a similar argument is applied for a supersolution.
It is enough to prove `only if' parts of both assertions since `if' parts are obvious.

(i) Set $\psi:=\phi+(u-\phi)(\hat{t},\hat{x})$.
Then $u-\psi$ attains a maximum at $(\hat{t},\hat{x})$ over $[a,b]\times\overline{B}$ and $(u-\psi)(\hat{t},\hat{x})=0$.
Since $u$ is a subsolution of \eqref{e:fhj}, 
$$
J_{\hat{t}-a}[\psi](\hat{t},\hat{x})+K_{\hat{t}-a}[u](\hat{t},\hat{x})+H(\hat{t},\hat{x},u(\hat{t},\hat{x}),D\psi(\hat{t},\hat{x}))\le0.
$$
It is easy to verify from the definition of $J_r[\phi]$ that $J_{\hat{t}-a}[\psi](\hat{t},\hat{x})=J_{\hat{t}-a}[\phi](\hat{t},\hat{x})$.
Clearly, $D\psi(\hat{t},\hat{x})=D\phi(\hat{t},\hat{x})$, so that \eqref{e:subineq} holds.

(ii) For $j\ge0$ we set $\phi_j(t,x):=\phi(t,x)+j^{-1}|t-\hat{t}|^2+|x-\hat{x}|^2$ on $(0,T]\times\mathbb{R}^d$.
Then $\phi_j\in\mathcal{C}^1([0,T]\times\mathbb{R}^d)$ and $u-\phi_j$ attains a maximum at $(\hat{t},\hat{x})$ over $[a,b]\times\overline{B}$.
Since $u$ is a subsolution of \eqref{e:fhj},
$$
J_{\hat{t}-a}[\phi_j](\hat{t},\hat{x})+K_{\hat{t}-a}[u](\hat{t},\hat{x})+H(\hat{t},\hat{x},u(\hat{t},\hat{x}),D\phi_j(\hat{t},\hat{x}))\le0.
$$
By the definition of $\phi_j$ we have
\begin{align*}
J_{\hat{t}-a}[\phi_j](\hat{t},\hat{x})=J_{\hat{t}-a}[\phi](\hat{t},\hat{x})-\frac{\alpha}{j\Gamma(1-\alpha)}\int_0^{\hat{t}}\tau^2\frac{d\tau}{\tau^{\alpha+1}}
\end{align*}
The last integral in the right-hand side is clearly finite, so vanishes as $j\to\infty$.
Since $D\phi_j(\hat{t},\hat{x})=D\phi(\hat{t},\hat{x})$, we reach \eqref{e:subineq}.

(iii) Choose $\delta>0$ so that $2\delta<\hat{t}-a$ and $\overline{B_{2\delta}(\hat{x})}\subset B$.
Let $\xi_1,\xi_2:[0,T]\times\mathbb{R}^d\to[0,1]$ be $C^{\infty}$ functions such that $\xi_1+\xi_2=1$ in $[0,T]\times\mathbb{R}^d$,
$\xi_1=1$ on $[\hat{t}-\delta,\hat{t}]\times B_{\delta}(\hat{x})$ and
$\xi_2=1$ on $([0,T]\times\mathbb{R}^d)\setminus([\hat{t}-2\delta,\hat{t}]\times B_{2\delta}(\hat{x}))$.
Set $\psi:=\xi_1\phi+\xi_2 M+(u-\phi)(\hat{t},\hat{x})$ on $[0,T]\times\mathbb{R}^d$, where $M:=\max_{\overline{Q_T}}u+1$.
Then $\psi\in\mathcal{C}^1([0,T]\times\mathbb{R}^d)$ and $u-\psi$ attains a zero maximum $(\hat{t},\hat{x})$ over $[\hat{t}-\delta,\hat{t}]\times\overline{B}\subset[a,b]\times\overline{B}$.
Since $u$ is a subsolution of \eqref{e:fhj},
$$
J_{\delta}[\psi](\hat{t},\hat{x})+K_{\delta}[u](\hat{t},\hat{x})+H(\hat{t},\hat{x},u(\hat{t},\hat{x}),D\psi(\hat{t},\hat{x}))\le0.
$$
It is easy that $J_{\delta}[\psi](\hat{t},\hat{x})=J_{\delta}[\phi](\hat{t},\hat{x})$ and $D\psi(\hat{t},\hat{x})=D\phi(\hat{t},\hat{x})$.
Moreover, since $u(\hat{t},\hat{x})-u(\hat{t}-\tau)\ge \phi(\hat{t},\hat{x})-\phi(\hat{t}-\tau,\hat{x})$ on $[0,\hat{t}-a]$,
\begin{align*}
K_{\delta}[u](\hat{t},\hat{x})\ge K_{\hat{t}-a}[u](\hat{t},\hat{x})+\frac{\alpha}{\Gamma(1-\alpha)}\int_{\delta}^{\hat{t}-a}(\phi(\hat{t},\hat{x})-\phi(\hat{t}-\tau,\hat{x}))\frac{d\tau}{\tau^{\alpha+1}}.
\end{align*}
Thus we have $J_{\delta}[\psi](\hat{t},\hat{x})+K_{\delta}[u](\hat{t},\hat{x})\ge J_{\hat{t}-a}[\phi](\hat{t},\hat{x})+K_{\hat{t}-a}[u](\hat{t},\hat{x})$, which is nothing but \eqref{e:subineq}.
\end{proof}

\begin{remark}\label{r:addconstant}
By the similar way in the proof of (i) it turns out that, if $u$ is a subsolution (resp. supersolution) of \eqref{e:fhj}, then $u-C$ (resp. $u+C$) is a subsolution (resp. supersolution) of \eqref{e:fhj} for any positive constant $C>0$.
This is valid even for sub/supersolutions of \eqref{e:fhj}-\eqref{e:init}.  
Here a proof needs (A3).
\end{remark}

{In what follows we establish an equivalent definition of solution.
The similar fact is well known for pdes with space-fractional derivatives and integer-order time derivative, and it is utilized to obtain meaningful observations; see \cite{ai}, \cite{bi}, \cite{cs} and \cite{i}.
Even in our case the following fact is effective in obtaining results, especially the comparison theorem.}

\begin{lemma}[Equivalence]
\label{l:equivalence}
A function $u\in USC(\overline{Q_T})$ (resp. $LSC(\overline{Q_T})$) is a subsolution (resp. supersolution) of \eqref{e:fhj} if and only if
$K_0[u](\hat{t},\hat{x})$ exists and
\begin{equation}
\label{e:0equation}
K_0[u](\hat{t},\hat{x})+H(\hat{t},\hat{x},u(\hat{t},\hat{x}),D\phi(\hat{t},\hat{x}))\le0\quad(\text{resp. $\ge0$})
\end{equation}
whenever $u-\phi$ attains a maximum (resp. minimum) at $(\hat{t},\hat{x})\in(0,T]\times\mathbb{R}^d$ over $[0,T]\times\mathbb{R}^d$ for $\phi\in\mathcal{C}^1([0,T]\times\mathbb{R}^d)$.
\end{lemma}

\begin{proof}
We only prove for subsolutions since the similar argument is applied for supersolutions.

We first prove the `if' part.
To do so, let $a,b\in[0,T]$ with $a<b$ and an open ball $B$ in $\mathbb{R}^d$ fix arbitrarily.
Assume that $u-\phi$ attains a maximum at $(\hat{t},\hat{x})\in (a,b]\times B$ over $[a,b]\times \overline{B}$ for $\phi\in\mathcal{C}^1((0,T]\times\mathbb{R}^d)$.
Define $\psi\in\mathcal{C}^1([0,T]\times\mathbb{R}^d)$ similarly as in the proof of Proposition \ref{p:replacement} (iii).
Then $u-\psi$ attains a zero maximum $(\hat{t},\hat{x})$ over $[0,T]\times\mathbb{R}^d$.
Thus $K_0[u](\hat{t},\hat{x})$ exists and 
\begin{equation}
\label{e:goal2}
K_0[u](\hat{t},\hat{x})+H(\hat{t},\hat{x},u(\hat{t},\hat{x}),D\psi(\hat{t},\hat{x}))\le0.
\end{equation}
The relationship between $u$ and $\phi$ implies that
$$
u(\hat{t},\hat{x})-u(\hat{t}-\tau,\hat{x})\ge \phi(\hat{t},\hat{x})-\phi(\hat{t}-\tau,\hat{x})
$$
for all $[0,\hat{t}-a]$, which further yields $J_{\hat{t}-a}[u](\hat{t},\hat{x})\ge J_{\hat{t}-a}[\phi](\hat{t},\hat{x})$.
Since $D\psi(\hat{t},\hat{x})=D\phi(\hat{t},\hat{x})$ and $K_0[u](\hat{t},\hat{x})=J_{\hat{t}-a}[u](\hat{t},\hat{x})+K_{\hat{t}-a}[u](\hat{t},\hat{x})$, the assertion follows immediately.

To  prove the `only if' part we assume that $u-\phi$ attains a maximum at $(\hat{t},\hat{x})\in(0,T]\times\mathbb{R}^d$ over $[0,T]\times\mathbb{R}^d$ for $\phi\in\mathcal{C}^1([0,T]\times\mathbb{R}^d)$.
Set $\psi:=\phi+(u-\phi)(\hat{t},\hat{x})$ on $(0,T]\times\mathbb{R}^d$.
Let $r>0$ be a parameter such that $\hat{t}-r>0$.
Then $u-\psi$ attains a zero maximum at $(\hat{t},\hat{x})$ over $[\hat{t}-r,\hat{t}]\times\overline{B(\hat{x})}$ for all $r$, where $B(\hat{x})$ is an open ball centered at $\hat{x}$ in $\mathbb{R}^d$. 
Since $u$ is a subsolution of \eqref{e:fhj},
\begin{equation}
\label{e:limitofr}
J_r[\psi](\hat{t},\hat{x})+K_r[u](\hat{t},\hat{x})+H(\hat{t},\hat{x},u(\hat{t},\hat{x}),D\psi(\hat{t},\hat{x}))\le0
\end{equation}
for all $r$.
From Proposition \ref{p:propertyintegral} and its proof we know that $J_r[\psi](\hat{t},\hat{x})$ and $K_r[u^-](\hat{t},\hat{x})$ exist for each $r$ and moreover $\lim_{r\to0}J_r[\psi](\hat{t},\hat{x})=0$. 
Thus it is enough to show that $K_r[u^+]$ exists for each small $r$ and $\lim_{r\to0}K_r[u^{\pm}]=K_0[u^{\pm}]$ exist as a finite number.
Indeed, if this is proved, it means that $K_0[u](\hat{t},\hat{x})$ exists and \eqref{e:0equation} follows by passing to the limit $r\to0$ in \eqref{e:limitofr}. 

Define a function $v_r:[0,T]\to\mathbb{R}$ by
\begin{equation*}
	v_r(\tau)=\begin{cases}
		\psi(\hat{t},\hat{x})-\psi(\hat{t}-\tau,\hat{x})\quad&\text{for $\tau\in[0,\hat{t}-r)\times\mathbb{T}^d$,}\\
		u(\hat{t},\hat{x})-u(\hat{t}-\tau,\hat{x})\quad&\text{for $\tau\in[\hat{t}-r,\hat{t}]\times\mathbb{T}^d$.}
	\end{cases}
\end{equation*}
We rewrite \eqref{e:limitofr} as
\begin{equation}
\label{e:integrableofvr}
I[v_r]\le \frac{\Gamma(1-\alpha)}{\alpha}\left(-H(\hat{t},\hat{x},u(\hat{t},\hat{x}),D\psi(\hat{t},\hat{x}))-\frac{u(\hat{t},\hat{x})-u(0,\hat{x})}{\hat{t}^{\alpha}\Gamma(1-\alpha)}\right)=:C,
\end{equation}
where
$$
I[v_r](\hat{t},\hat{x})=\int_0^{\hat{t}}v_r(\tau)\frac{d\tau}{\tau^{\alpha+1}}.
$$
From the relationship between $u$ and $\psi$ we see
\begin{equation}
\label{e:uandpsi}
\psi(\hat{t},\hat{x})-\psi(\hat{t}-\tau,\hat{x})\le u(\hat{t},\hat{x})-u(\hat{t}-\tau,\hat{x})
\end{equation}
on $[0,\hat{t}-r]$.
Then it suffices to prove that $I[v_r^+]$ exists for each small $r$ and $\lim_{r\to0}I[v_r^{\pm}]=I[v_0^{\pm}]$ exists as a finite number. 

By the definition of $v_r$ and \eqref{e:uandpsi}, $v_r^+$ is monotone increasing with respect to $r$ in the sense that $v_{r_1}^+\le v_{r_2}^+$ on $[0,\hat{t}]$ if $r_{1}\ge r_2$.
The monotone convergence theorem implies that$ \lim_{r\to0}I[v_r^+]=I[v_0^+]$.
It is verified similarly as for $v_r^+$ that $v_r^-$ is monotone decreasing with respect to $r$ in the sense that $v_{r_1}^-\le v_{r_2}^-$ on $[0,\hat{t}]$ if $r_1\le r_2$.
Thus we have from \eqref{e:integrableofvr}
\begin{equation}
\label{e:conclusion}
I[v_{r_1}^+](\hat{t},\hat{x})\le I[v_{r_1}^-](\hat{t},\hat{x})+C\le I[v_{r_2}^-](\hat{t},\hat{x})+C.
\end{equation}
This implies that $I[v_r^+]$ exists for each small $r$ and $I[v_0^+]$ exists (as a finite number) by passing to the limit $r_1\to0$.
The monotone convergence theorem for $I[v_r^-]$ implies that $\lim_{r\to0}I[v_r^-]=I[v_0^-]$.
Therefore \eqref{e:conclusion} ensures that $I[v_0^-]$ exists (as a finite number).
The proof is now complete.

\end{proof}

\begin{proposition}[Consistency]
\label{p:consistency}
Assume that $u\in \mathcal{C}^{1}(\overline{Q_T})$.
Then $u$ is a classical solution of \eqref{e:fhj}-\eqref{e:init} if and only if $u$ is a viscosity solution of \eqref{e:fhj}-\eqref{e:init}.
\end{proposition}

\begin{proof}
Assume that $u$ is a viscosity subsolution.
We may take $\phi\equiv u$ so that $u-\phi$ attains a maximum at every point in $Q_T$.
Since $u$ is a viscosity subsolution of \eqref{e:fhj}, Lemma \ref{l:equivalence} implies that
\begin{equation}
\label{e:classical}
K_0[u](t,x)+H(t,x,u(t,x),Du(t,x))\le0
\end{equation}
for all $(t,x)\in Q_T$.
Similarly, we have the reverse inequality of \eqref{e:classical} from an inequality by viscosity supersolution. 
This shows that $u$ is a classical solution since $K_0[u]=\partial_t^{\alpha}u$ by Proposition \ref{p:integrationbyparts} (ii).

On the contrary we assume that $u$ is a classical solution and that $u-\phi$ attains a maximum at $(\hat{t},\hat{x})\in(a,b]\times B$ over $[a,b]\times\overline{B}$ for $\phi\in \mathcal{C}^{1}([0,T]\times\mathbb{R}^d)$,
where $a,b\in(0,T]$ are constants and $B$ is an open ball in $\mathbb{R}^d$.
Since $(\partial_t^\alpha u)(\hat{t},\hat{x})=K_0[u](\hat{t},\hat{x})\ge J_{\hat{t}-a}[\phi](\hat{t},\hat{x})+K_{\hat{t}-a}[u](\hat{t},\hat{x})$,
to combine the maximum principle in space implies that $u$ is a viscosity subsolution.
It is similar for viscosity supersolutions.

Since an initial condition is easily verified, we obtain the conclusion.
\end{proof}

\section{Comparison Principle}\label{s:comp}
\begin{theorem}[Comparison principle]
\label{t:comp}
Assume that (A1)-(A3).
Let $u\in USC(\overline{Q_T})$ and $v\in LSC(\overline{Q_T})$ be a subsolution and a supersolution of \eqref{e:fhj}, respectively.
If $u(0,\cdot)\le v(0,\cdot)$ on $\mathbb{T}^{d}$, then $u\le v$ on $\overline{Q_T}$.
\end{theorem}

We shall prepare one lemma for a proof of Theorem \ref{t:comp}; see {\cite[Lemma 2]{cl2}}, \cite[Lemma 3.3]{ishii} and \cite[Lemma 1]{ishii2} for similar results for $\alpha=1$.
To do so we invoke a limit inferior/superior inequality of product of constant sequences that one of sequences is allowed to be negative.
The statement looks fundamental and the proof is standard but we give for the reader's convenience.

\begin{proposition}
\label{p:semicontinuityofproduct}
Let $\{f_{\varepsilon}\}_{\varepsilon>0}$ and $\{g_{\varepsilon}\}_{\varepsilon>0}$ be constant sequences.
Assume that $g_{\varepsilon}$ is nonnegative, $\liminf_{\varepsilon\to0}f_{\varepsilon}\ge f_0$ (resp. $\limsup_{\varepsilon\to0}f_{\varepsilon}\le f_0$) and $\lim_{\varepsilon\to0}g_{\varepsilon}=g_0$ for some constants $f_0$ and $g_0$.
Then
$$
\liminf_{\varepsilon\to0}f_{\varepsilon}g_{\varepsilon}\ge f_0g_0,\quad
(\text{resp. $\limsup_{\varepsilon\to0}f_{\varepsilon}g_{\varepsilon}\le f_0g_0$}.)
$$
\end{proposition}

\begin{proof}
It is enough to prove the case when $\liminf_{\varepsilon\to0}f_{\varepsilon}\ge f_0$ since another case is proved by changing a sign of $f_{\varepsilon}$.
Then for any $\delta>0$ there exists $\varepsilon_{\delta}>0$ such that $f_{\varepsilon}\ge f_0-\delta$ for all $\varepsilon<\varepsilon_{\delta}$.
It is fundamental that
$$
\liminf_{\varepsilon\to0}h_{\varepsilon}g_{\varepsilon}\ge \liminf_{\varepsilon\to0}h_{\varepsilon}\liminf_{\varepsilon\to0}g_{\varepsilon}
$$
for nonnegative constants $h_{\varepsilon},g_{\varepsilon}$.
Applying this fact as $h_{\varepsilon}:=f_{\varepsilon}-f_0-\delta(\ge0)$ we see
\begin{align*}
\liminf_{\varepsilon\to0}f_{\varepsilon}g_{\varepsilon}
&\ge\liminf_{\varepsilon\to0}(f_{\varepsilon}-f_0-\delta)g_{\varepsilon}+\liminf_{\varepsilon\to0}(f_0+\delta)g_{\varepsilon}\\
&\ge -\delta g_0+(f_0+\delta)g_0=f_0g_0.
\end{align*}
\end{proof}

\begin{lemma}\label{l:ishiilemma}
Assume (A1).
Let $u\in USC(\overline{Q_T})$ and $v\in LSC(\overline{Q_T})$ be a subsolution and a supersolution of \eqref{e:fhj}, respectively.
Assume that $(t,x,y)\mapsto u(t,x)-v(t,y)-\phi(t,x,y)$ attains a maximum at $(\hat{t},\hat{x},\hat{y})\in (0,T]\times\mathbb{R}^d\times\mathbb{R}^d$
 over $[0,T]\times\mathbb{R}^d\times\mathbb{R}^d$ for $\phi\in\mathcal{C}^1([0,T]\times\mathbb{R}^d\times\mathbb{R}^d)$.
Then
$$
K_0[u](\hat{t},\hat{x})-K_0[v](\hat{t},\hat{y})+H(\hat{t},\hat{x},u(\hat{t},\hat{x}),D_x\phi(\hat{t},\hat{x},\hat{y}))-H(\hat{t},\hat{y},v(\hat{t},\hat{y}),-D_y\phi(\hat{t},\hat{x},\hat{y}))\le0.
$$
\end{lemma}

\begin{proof}
We shall show that there exists a constant $C_r>0$ such that $C_r\to0$ as $r\to0$ and
\begin{equation}
\label{e:mostlygoal}
\begin{aligned}
&-C_r+J_r[\phi](\hat{t},\hat{x},\hat{y})+K_r[u](\hat{t},\hat{x})+K_r[v](\hat{t},\hat{y})\\
&+H(\hat{t},\hat{x},u(\hat{t},\hat{x}),D_x\phi(\hat{t},\hat{x},\hat{y}))-H(\hat{t},\hat{y},v(\hat{t},\hat{y}),-D_y\phi(\hat{t},\hat{x},\hat{y}))\le0
\end{aligned}
\end{equation}
for all $r\in(0,\hat{t})$.
If this is clarified, passing to the limit $r\to0$ in \eqref{e:mostlygoal} yields the desired result by repeating the `only if' in the proof of Lemma \eqref{l:equivalence}.
Henceforth, let $r\in(0,\hat{t})$ fix arbitrarily.

For $\varepsilon>0$ we consider a function $\Phi:[0,T]\times[0,T]\times\mathbb{R}^d\times\mathbb{R}^d\to\mathbb{R}$ defined by
$$
\Phi(t,s,x,y)=u(t,x)-v(s,y)-\phi(t,x,y)-\frac{|t-s|^2}{2\varepsilon}-|t-\hat{t}|^2-|x-\hat{x}|^2-|y-\hat{y}|^2.
$$
Since $\Phi\to-\infty$ as $|x|,|y|\to+\infty$ and $\Phi$ is bounded from above, it attains a maximum at a point $(t_{\varepsilon},s_{\varepsilon},x_{\varepsilon},y_{\varepsilon})\in[0,T]\times[0,T]\times\mathbb{R}^d\times\mathbb{R}^d$.
By following the standard argument of the theory of viscosity solutions we obtain
\begin{equation}
\label{e:wellknown}
	\begin{cases}
		(t_{\varepsilon},s_{\varepsilon},x_{\varepsilon},y_{\varepsilon})\to(\hat{t},\hat{t},\hat{x},\hat{y}),\\
		u(t_{\varepsilon},x_{\varepsilon})\to u(\hat{t},\hat{x})\text{ and $v(s_{\varepsilon},y_{\varepsilon})\to v(\hat{t},\hat{y})$}
	\end{cases}
\end{equation}
as $\varepsilon\to0$ by taking a subsequence if necessary; see \cite[Theorem II.3.1]{bcd} and \cite[Lemma 3.1]{cil} for detail.
Note that $t_{\varepsilon}>0$ for sufficiently small $\varepsilon$ since $\hat{t}>0$.

For such a small parameter $\varepsilon$, $(t,x)\mapsto\Phi(t,s_{\varepsilon},x,y_{\varepsilon})$ attains a maximum at $(t_{\varepsilon},x_{\varepsilon})\in (0,T]\times\mathbb{R}^d$ over $[0,T]\times\mathbb{R}^d$ and
 $(s,y)\mapsto-\Phi(t_{\varepsilon},s,x_{\varepsilon},y)$ attains a minimum at $(s_{\varepsilon},y_{\varepsilon})\in (0,T]\times\mathbb{R}^d$ over $[0,T]\times\mathbb{R}^d$.
Since $u$ and $v$ are respectively a subsolution and a supersolution of \eqref{e:fhj}, Lemma \ref{l:equivalence} implies that $K_0[u](t_{\varepsilon},x_{\varepsilon}),K_0[v](s_{\varepsilon},y_{\varepsilon})$ exist for each $\varepsilon$ and 
\begin{align}
K_0[u](t_{\varepsilon},x_{\varepsilon})+H(t_{\varepsilon},x_{\varepsilon},u(t_{\varepsilon},x_{\varepsilon}),D_x\phi(t_{\varepsilon},x_{\varepsilon},y_{\varepsilon})+2(x_{\varepsilon}-\hat{x}))\le0,\label{e:sub}\\
K_0[v](s_{\varepsilon},y_{\varepsilon})+H(s_{\varepsilon},y_{\varepsilon},v(s_{\varepsilon},y_{\varepsilon}),-D_y\phi(s_{\varepsilon},x_{\varepsilon},y_{\varepsilon})-2(y_{\varepsilon}-\hat{y}))\ge0.\label{e:super}
\end{align}
Thus, by subtracting \eqref{e:super} from \eqref{e:sub}, we see
\begin{equation}
\begin{aligned}
\label{e:takenlimitinf}
&K_0[u](t_{\varepsilon},x_{\varepsilon})-K_0[v](s_{\varepsilon},y_{\varepsilon})\\
&+H(t_{\varepsilon},x_{\varepsilon},u(t_{\varepsilon},x_{\varepsilon}),D_x\phi(t_{\varepsilon},x_{\varepsilon},y_{\varepsilon})+2(x_{\varepsilon}-\hat{x}))\\
&-H(s_{\varepsilon},y_{\varepsilon},v(s_{\varepsilon},y_{\varepsilon}),-D_y\phi(s_{\varepsilon},x_{\varepsilon},y_{\varepsilon})-2(y_{\varepsilon}-\hat{y}))\le0
\end{aligned}
\end{equation}
for each $\varepsilon$.

We shall pass to the limit $\varepsilon\to0$ in \eqref{e:takenlimitinf}.
For Hamiltonians it is easily seen thanks to (A1), \eqref{e:wellknown} and the smoothness of $\phi$ that
\begin{align*}
&H(t_{\varepsilon},x_{\varepsilon},u(t_{\varepsilon},x_{\varepsilon}),D_x\phi(t_{\varepsilon},x_{\varepsilon},y_{\varepsilon})+2(x_{\varepsilon}-\hat{x}))\\
&-H(s_{\varepsilon},y_{\varepsilon},v(s_{\varepsilon},y_{\varepsilon}),-D_y\phi(s_{\varepsilon},x_{\varepsilon},y_{\varepsilon})-2(y_{\varepsilon}-\hat{y}))\\
&\to H(\hat{t},\hat{x},u(\hat{t},\hat{x}),D_x\phi(\hat{t},\hat{x},\hat{y}))-H(\hat{t},\hat{y},v(\hat{t},\hat{y}),-D_y\phi(\hat{t},\hat{x},\hat{y}))
\end{align*}
as $\varepsilon\to0$.
Let us focus on $K_0[u](t_{\varepsilon},x_{\varepsilon})-K_0[v](s_{\varepsilon},y_{\varepsilon})$.
Assume hereafter that $\varepsilon$ is so small that $r<\min\{t_{\varepsilon},s_{\varepsilon}\}$ for all $\varepsilon$, which is possible since $(t_{\varepsilon},s_{\varepsilon})\to(\hat{t},\hat{t})$ as $\varepsilon\to0$ (see \eqref{e:wellknown}).
Set
\begin{align*}
&I_{1,\varepsilon}:=\frac{u(t_{\varepsilon},x_{\varepsilon})-u(0,x_{\varepsilon})}{t_{\varepsilon}^{\alpha}\Gamma(1-\alpha)}-\frac{v(s_{\varepsilon},y_{\varepsilon})-v(0,y_{\varepsilon})}{s_{\varepsilon}^{\alpha}\Gamma(1-\alpha)},\\
&I_{2,\varepsilon}:=\int_0^r(u(t_{\varepsilon},x_{\varepsilon})-u(t_{\varepsilon}-\tau,x_{\varepsilon}))\frac{d\tau}{\tau^{\alpha+1}}
-\int_0^r(v(s_{\varepsilon},y_{\varepsilon})-v(s_{\varepsilon}-\tau,y_{\varepsilon}))\frac{d\tau}{\tau^{1+\alpha}},
\end{align*}
and
$$
I_{3,\varepsilon}:=\int_r^{t_{\varepsilon}}(u(t_{\varepsilon},x_{\varepsilon})-u(t_{\varepsilon}-\tau,x_{\varepsilon}))\frac{d\tau}{\tau^{\alpha+1}}
-\int_r^{s_{\varepsilon}}(v(s_{\varepsilon},y_{\varepsilon})-v(s_{\varepsilon}-\tau,y_{\varepsilon}))\frac{d\tau}{\tau^{1+\alpha}}
$$
so that $K_0[u](t_{\varepsilon},x_{\varepsilon})-K_0[v](s_{\varepsilon},y_{\varepsilon})=I_{3,\varepsilon}+\alpha(I_{1,\varepsilon}+I_{2,\varepsilon})/\Gamma(1-\alpha)$.

First, for $I_{1,\varepsilon}$, Proposition \ref{p:semicontinuityofproduct} with $f_{\varepsilon}:=u(t_{\varepsilon},x_{\varepsilon})-u(0,x_{\varepsilon})-v(s_{\varepsilon},y_{\varepsilon})+v(0,y_{\varepsilon})$ and $g_{\varepsilon}:=(t_{\varepsilon}^{\alpha}\Gamma(1-\alpha))^{-1}$ implies that
\begin{equation}
\label{e:I1}
\liminf_{\varepsilon\to0}I_{1,\varepsilon}\ge \frac{u(\hat{t},\hat{x})-u(0,\hat{x})}{\hat{t}^{\alpha}\Gamma(1-\alpha)}-\frac{v(\hat{t},\hat{y})-v(0,\hat{y})}{\hat{t}^{\alpha}\Gamma(1-\alpha)}.
\end{equation}

Next, since 
\begin{align*}
&u(t_{\varepsilon},x_{\varepsilon})-u(t_{\varepsilon}-\tau,x_{\varepsilon})-(v(s_{\varepsilon},y_{\varepsilon})-v(s_{\varepsilon}-\tau,y_{\varepsilon}))\\
&\ge \phi(t_{\varepsilon},x_{\varepsilon},y_{\varepsilon})-\phi(t_{\varepsilon}-\tau,x_{\varepsilon},y_{\varepsilon})+|t_{\varepsilon}-\hat{t}|^2-|t_{\varepsilon}-\tau-\hat{t}|^2
\end{align*}
for all $\tau\in[0,r]$ by the inequality $\Phi(t_{\varepsilon},s_{\varepsilon},x_{\varepsilon},y_{\varepsilon})\ge\Phi(t_{\varepsilon}-\tau,s_{\varepsilon}-\tau,x_{\varepsilon},y_{\varepsilon})$,
we see
$$
\frac{\alpha}{\Gamma(1-\alpha)}I_{2,\varepsilon}\ge J_r[\phi](t_{\varepsilon},x_{\varepsilon},y_{\varepsilon})+J_r[|t_{\varepsilon}-\hat{t}-\cdot|^2](t_{\varepsilon}).
$$
Proposition \ref{p:propertyintegral} (iv) ensures that $\lim_{\varepsilon\to0}J_r[\phi](t_{\varepsilon},x_{\varepsilon},y_{\varepsilon})=J_r[\phi](\hat{t},\hat{x},\hat{y})$.
Besides, $J_r[|t_{\varepsilon}-\hat{t}-\cdot|^2](t_{\varepsilon})$ can be calculated precisely as 
\begin{align*}
J_r[|t_{\varepsilon}-\hat{t}-\cdot|^2](t_{\varepsilon})
&=\frac{\alpha}{\Gamma(1-\alpha)}\int_0^r(|t_{\varepsilon}-\hat{t}|^2-|t_{\varepsilon}-\hat{t}-\tau|^2)\frac{d\tau}{\tau^{1+\alpha}}\\
&=\frac{\alpha}{\Gamma(1-\alpha)}\int_0^r(2(t_{\varepsilon}-\hat{t})\tau-\tau^2)\frac{d\tau}{\tau^{1+\alpha}}\\
&=\frac{\alpha}{\Gamma(1-\alpha)}\left(\frac{2(t_{\varepsilon}-\hat{t})r^{1-\alpha}}{1-\alpha}-\frac{r^{2-\alpha}}{2-\alpha}\right).
\end{align*}
Hence
$$
\lim_{\varepsilon\to0}J_r[|t_{\varepsilon}-\hat{t}-\cdot|^2](t_{\varepsilon})=-\frac{\alpha r^{2-\alpha}}{(2-\alpha)\Gamma(1-\alpha)}=:-C_r.
$$
Note that $C_r\to0$ as $r\to0$.
Therefore we know for $I_{2,\varepsilon}$ that
\begin{equation}
\label{e:I2}
\liminf_{\varepsilon\to0}\frac{\alpha}{\Gamma(1-\alpha)}I_{2,\varepsilon}\ge -C_r+J_r[\phi](\hat{t},\hat{x},\hat{y}).
\end{equation}

Finally, for $I_{3,\varepsilon}$, we first see an existence of constants $C_1,C_2>0$ independent of $\varepsilon$ such that
\begin{equation}
\label{e:integrability}
\begin{aligned}
&(u(t_{\varepsilon},x_{\varepsilon})-u(t_{\varepsilon}-\tau,x_{\varepsilon}))\mathds{1}_{(r,t_{\varepsilon})}(\tau)\ge -C_1\mathds{1}_{(r,T)}(\tau),\\
&(v(s_{\varepsilon},y_{\varepsilon})-v(s_{\varepsilon}-\tau,y_{\varepsilon}))\mathds{1}_{(r,s_{\varepsilon})}(\tau)\le C_2\mathds{1}_{(r,T)}(\tau)
\end{aligned}
\end{equation}
on $[0,T]$.
Indeed, since $\lim_{\varepsilon\to0}u(t_{\varepsilon},x_{\varepsilon})=u(\hat{t},\hat{x})$, there is a constant $C>0$ independent of $\varepsilon$
\begin{align*}
(u(t_{\varepsilon},x_{\varepsilon})-u(t_{\varepsilon}-\tau,x_{\varepsilon}))\mathds{1}_{(r,t_{\varepsilon})}(\tau)
&\ge (u(\hat{t},\hat{x})-C-\max_{\overline{Q_T}}u)\mathds{1}_{(r,t_{\varepsilon})}(\tau)\\
&\ge -|u(\hat{t},\hat{x})-C-\max_{\overline{Q_T}}u|\mathds{1}_{(r,T)}.
\end{align*}
This shows the above one of \eqref{e:integrability} and another is proved similarly.
Note that both right-hand sides of \eqref{e:integrability} multiplied by $\tau^{-\alpha-1}$ is integrable on $[0,T]$.
Proposition \ref{p:semicontinuityofproduct} implies that
\begin{align*}
&\liminf_{\varepsilon\to0}(u(t_{\varepsilon},x_{\varepsilon})-u(t_{\varepsilon}-\cdot,x_{\varepsilon}))\mathds{1}_{(r,t_{\varepsilon})}(\cdot)
\ge(u(\hat{t},\hat{x})-u(\hat{t}-\cdot,\hat{x}))\mathds{1}_{(r,\hat{t})}(\cdot),\\
&\limsup_{\varepsilon\to0}(v(s_{\varepsilon},y_{\varepsilon})-v(s_{\varepsilon}-\cdot,y_{\varepsilon}))\mathds{1}_{(r,s_{\varepsilon})}(\cdot)
\le(v(\hat{t},\hat{x})-v(\hat{t}-\cdot,\hat{x}))\mathds{1}_{(r,\hat{t})}(\cdot)
\end{align*}
for each $\tau\in(0,T)$.
Thus Fatou's lemma yields
\begin{equation}
\label{e:I3}
\liminf_{\varepsilon\to0}I_{3,\varepsilon}\ge \int_r^{\hat{t}}(u(\hat{t},\hat{x})-u(\hat{t}-\tau,\hat{x}))\frac{d\tau}{\tau^{1+\alpha}}
-\int_r^{\hat{t}}(v(\hat{t},\hat{y})-v(\hat{t}-\tau,\hat{y}))\frac{d\tau}{\tau^{1+\alpha}}.
\end{equation}

Summing up \eqref{e:I1}, \eqref{e:I2} and \eqref{e:I3} we reach
$$
\liminf_{\varepsilon\to0}(K_0[u](t_{\varepsilon},x_{\varepsilon})-K_0[v](s_{\varepsilon},y_{\varepsilon}))\ge -C_r+J_r[\phi](\hat{t},\hat{x},\hat{y})+K_r[u](\hat{t},\hat{x})-K_r[v](\hat{t},\hat{x}).
$$
Consequently, taking the limit inferior to both sides of \eqref{e:takenlimitinf} yields the desired inequality \eqref{e:mostlygoal}.
\end{proof}

\begin{proof}[Proof of Theorem \ref{t:comp}]
Suppose that the conclusion were false: $\max_{\overline{Q_T}}(u-v)=:\theta>0$.
For $\varepsilon>0$ we consider a function $\Phi:[0,T]\times\mathbb{R}^{d}\times\mathbb{R}^{d}\to\mathbb{R}$ defined by
$$
\Phi(t,x,y):=u(t,x)-v(t,y)-\frac{|x-y|^{2}}{2\varepsilon}.
$$

Let $(t_{\varepsilon},x_{\varepsilon},y_{\varepsilon})\in[0,T]\times\mathbb{R}^{d}\times\mathbb{R}^{d}$ be a maximum point of $\Phi$.
Then there is a point $(\hat{t},\hat{x})\in(0,T]\times\mathbb{R}^d$ such that
\begin{equation}
\label{e:tool0}
	\begin{cases}
		(t_{\varepsilon},x_{\varepsilon},y_{\varepsilon})\to(\hat{t},\hat{x},\hat{x}),\\
		|x_{\varepsilon}-y_{\varepsilon}|^{2}/\varepsilon\to0,\\
		u(t_{\varepsilon},x_{\varepsilon})\to u(\hat{t},\hat{x})\text{ and $v(s_{\varepsilon},y_{\varepsilon})\to v(\hat{t},\hat{x})$.}
	\end{cases}
\end{equation}
 as $\varepsilon\to0$ by taking a subsequence if necessary; see, e.g., \cite[Theorem II.3.1]{bcd}.
The above permits to use Lemma \ref{l:ishiilemma} and we know that $K_0[u](t_{\varepsilon},x_{\varepsilon}),K_0[v](t_{\varepsilon},y_{\varepsilon})$ exists for each $\varepsilon$ and
\begin{equation}
\label{e:goal}
K_0[u](t_{\varepsilon},x_{\varepsilon})-K_0[v](t_{\varepsilon},y_{\varepsilon})+H(t_{\varepsilon},x_{\varepsilon},u(t_{\varepsilon},x_{\varepsilon}),p_{\varepsilon})-H(t_{\varepsilon},y_{\varepsilon},v(t_{\varepsilon},y_{\varepsilon}),p_{\varepsilon})\le0.
\end{equation}
Here $p_{\varepsilon}=(x_{\varepsilon}-y_{\varepsilon})/\varepsilon$.

Since $u(t_{\varepsilon},x_{\varepsilon})-u(t_{\varepsilon}-\cdot,x_{\varepsilon})-v(t_{\varepsilon},y_{\varepsilon})+v(t_{\varepsilon}-\cdot,y_{\varepsilon})\ge0$ on $[0,t_{\varepsilon}]$ by the inequality $\Phi(t_{\varepsilon},x_{\varepsilon},y_{\varepsilon})\ge\Phi(t_{\varepsilon}-\tau,x_{\varepsilon},y_{\varepsilon})$, the term of integration in $K_0[u](t_{\varepsilon},x_{\varepsilon})-K_0[v](t_{\varepsilon},y_{\varepsilon})$ is estimated from below by zero, that is,
$$
K_0[u](t_{\varepsilon},x_{\varepsilon})-K_0[v](t_{\varepsilon},y_{\varepsilon})
\ge\frac{u(t_{\varepsilon},x_{\varepsilon})-v(t_{\varepsilon},y_{\varepsilon})-u(0,x_{\varepsilon})+v(0,y_{\varepsilon})}{t_{\varepsilon}^{\alpha}\Gamma(1-\alpha)}
$$
Since $u(t_{\varepsilon},x_{\varepsilon})>v(t_{\varepsilon},y_{\varepsilon})$ by the inequality $\Phi(t_{\varepsilon},x_{\varepsilon},y_{\varepsilon})\ge\theta>0$, Hamiltonians in \eqref{e:goal} are estimated as
\begin{align*}
&H(t_{\varepsilon},x_{\varepsilon},u(t_{\varepsilon},x_{\varepsilon}),p_{\varepsilon})-H(t_{\varepsilon},y_{\varepsilon},v(t_{\varepsilon},y_{\varepsilon}),p_{\varepsilon})\\
&\ge H(t_{\varepsilon},x_{\varepsilon},v(t_{\varepsilon},y_{\varepsilon}),p_{\varepsilon})-H(t_{\varepsilon},y_{\varepsilon},v(t_{\varepsilon},y_{\varepsilon}),p_{\varepsilon})\ge-\omega(|x_{\varepsilon}-y_{\varepsilon}|(1+|p_{\varepsilon}|))
\end{align*}
by (A2) and (A3).
From these, \eqref{e:goal} is led to
$$
\frac{u(t_{\varepsilon},x_{\varepsilon})-v(t_{\varepsilon},y_{\varepsilon})-u(0,x_{\varepsilon})+v(0,y_{\varepsilon})}{t_{\varepsilon}^{\alpha}\Gamma(1-\alpha)}\le\omega(|x_{\varepsilon}-y_{\varepsilon}|(1+|p_{\varepsilon}|)).
$$
Taking the limit inferior $\varepsilon\to0$ implies that
$$
\frac{\theta-u(0,\hat{x})+v(0,\hat{x})}{\hat{t}^{\alpha}\Gamma(1-\alpha)}\le0
$$
by Proposition \ref{p:semicontinuityofproduct}.
Since $u(0,\cdot)\le v(0,\cdot)$ on $\mathbb{T}^{d}$ and $\theta>0$, this is a contradiction.
\end{proof}

\begin{corollary}[Uniqueness]
\label{c:unique}
Assume (A1)-(A4).
Let $u\in C(\overline{Q_T})$ and $v\in C(\overline{Q_T})$ be solutions of \eqref{e:fhj}.
Then
\begin{equation}
\label{e:comp}
\max_{(t,x)\in\overline{Q_T}}|u(t,x)-v(t,x)|\le\max_{x\in\mathbb{T}^d}|u(0,x)-v(0,x)|.
\end{equation}
Moreover, if $u$ and $v$ are solutions of \eqref{e:fhj}-\eqref{e:init}, then $u\equiv v$ on $\overline{Q_T}$.
\end{corollary}

\begin{proof}
It suffices to prove \eqref{e:comp}.
Set $C:=\max_{x\in\mathbb{T}^d}|u(0,x)-v(0,x)|$.
Then $v-C$ and $v+C$ are a subsolution and a supersolution of \eqref{e:fhj}, respectively; see Remark \ref{r:addconstant}.
Moreover 
$$
v(0,\cdot)-C\le u(0,\cdot)\le v(0,\cdot)+C\quad\text{on $\mathbb{T}^d$}
$$
by the definition of $C$.
Thus, from Theorem \ref{t:comp}, we have $|u-v|\le C$ on $\overline{Q_T}$.
The proof is complete by taking the maximum over $\overline{Q_T}$ to both sides.
\end{proof}

For the reader's convenience we give a statement of the comparison principle for a general bounded domain $\Omega$ without a proof.

\begin{theorem}
Let $\Omega$ be a bounded domain in $\mathbb{R}^d$.
Let $u\in USC([0,T]\times\overline{\Omega};\mathbb{R})$ and $v\in LSC([0,T]\times\overline{\Omega};\mathbb{R})$ be a subsolution and a supersolution of \eqref{e:fhj} on $(0,T]\times\overline{\Omega}$, respectively.
If $u\le v$ on $(\{0\}\times\overline{\Omega})\cup([0,T]\times\partial\Omega)$, then $u\le v$ on $[0,T]\times\overline{\Omega}$.
\end{theorem}

\section{Existence result}

Let denote by $S^{-}$ and $S^{+}$ a set of upper semicontinuous subsolutions and lower semicontinuous supersolutions of \eqref{e:fhj}, respectively.
Note that $S^\pm\neq\emptyset$ as will be observed in Corollary \ref{cor:constructbarrier} later.

\begin{lemma}[Closedness under supremum/infimum operator]
\label{l:closedness}
Assume (A1).
Let $X$ be a nonempty subset of $S^{-}$ (resp. $S^{+}$).
Define 
$$
u(t,x):=\sup\{ v(t,x) \mid v\in X\}\quad\text{(resp. $\inf\{v(t,x) \mid v\in X\}$)}
$$
for $(t,x)\in \overline{Q_T}$.
Assume that $u^*<+\infty$ (resp. $u_*>-\infty$) on $\overline{Q_T}$.
Then $u^*$ (resp. $u_*$) is a subsolution (resp. supersolution) of \eqref{e:fhj}.
\end{lemma}

\begin{proof}
We only prove for a subsolution since the argument for a supersolution is similar.
Fix $[a,b]\times B\subset(0,T]\times\mathbb{R}^d$ arbitrarily, where $a<b$ and $B$ is an open in $\mathbb{R}^d$.
Assume that $u^*-\phi$ attains a maximum at $(\hat{t},\hat{x})\in (a,b]\times B$ over $[a,b]\times\overline{B}$ for $\phi\in \mathcal{C}^1([0,T]\times\mathbb{R}^d)$.
Then we must show that
\begin{equation}
\label{e:closedgoal}
J_{\hat{t}-a}[\phi](\hat{t},\hat{x})+K_{\hat{t}-a}[u^*](\hat{t},\hat{x})+H(\hat{t},\hat{x},u^*(\hat{t},\hat{x}),D\phi(\hat{t},\hat{x}))\le0.
\end{equation}
By Proposition \ref{p:replacement} we may assume that $(\hat{t},\hat{x})$ is a strict maximum point of $u^*-\phi$ such that $(u^*-\phi)(\hat{t},\hat{x})=0$.

By arguing similarly as for $\alpha=1$ we find sequences $\{(t_j,x_j)\}_{j\ge0}$ and $\{v_j\}_{j\ge0}\subset X$ such that, for each $j\ge0$, $v_j-\phi$ attains a maximum at $(t_j,x_j)\in(a,b]\times B$ over $[a,b]\times\overline{B}$ and $(t_j,x_j,v_j(t_j,x_j))\to (\hat{t},\hat{x},u^*(\hat{t},\hat{x}))$ as $j\to\infty$.
Indeed it is enough to translate slightly the proof of \cite[Lemma 2.4.1]{g} to the current situation.
This is not difficult, so the detail is safely omitted.
Since $v_j$ is a subsolution of \eqref{e:fhj},
\begin{equation}
\label{e:jequation}
J_{t_j-a}[\phi](t_j,x_j)+K_{t_j-a}[v_j](t_j,x_j)+H(t_j,x_j,v_j(t_j,x_j),D\phi(t_j,x_j))\le0
\end{equation}
for each $j\ge0$.

We shall pass to the limit $j\to\infty$ in \eqref{e:jequation}.
The continuity of Hamiltonian (A1) ensures that
$$
\lim_{j\to\infty}H(t_j,x_j,v_j(t_j,x_j),D\phi(t_j,x_j))=H(\hat{t},\hat{x},u^*(\hat{t},\hat{x}),D\phi(\hat{t},\hat{x})).
$$
Proposition \ref{p:propertyintegral} implies that
$$
\lim_{j\to\infty}J_{t_j-a}[\phi](t_j,x_j)=J_{\hat{t}-a}[\phi](\hat{t},\hat{x}).
$$
Henceforth, let us focus on $K_{t_j-a}[v_j](t_j,x_j)$.
Since $v_j\le u\le u^*$ on $\overline{Q_T}$ by the definition of $u$ and $u^*$, Proposition \ref{p:semicontinuityofproduct} implies that
\begin{equation}
\label{e:nonintegration}
\begin{aligned}
\liminf_{j\to\infty}\frac{v_j(t_j,x_j)-v_j(0,x_j)}{t_j^{\alpha}\Gamma(1-\alpha)}
&\ge\liminf_{j\to\infty}\frac{v_j(t_j,x_j)-u^*(0,x_j)}{t_j^{\alpha}\Gamma(1-\alpha)}\\
&\ge\frac{u^*(\hat{t},\hat{x})-u^*(0,\hat{x})}{\hat{t}^{\alpha}\Gamma(1-\alpha)}.
\end{aligned}
\end{equation}
To handle the term of integration we first see the existence of a constant $C_2>0$ independent of $j$ such that
$$
(v_j(t_j,x_j)-v_j(t_j-\cdot,x_j))\mathds{1}_{[t_j-a,t_j]}(\cdot)\ge -C_2\mathds{1}_{[r,T]}(\cdot)
$$
on $[0,T]$ for sufficiently large $j$, where $r:=\min_{j\ge0}(t_j-a)>0$.
Indeed, since $\sup_{j\ge0}v_j\le u\le u^*$ and $u^*<+\infty$ on $\overline{Q_T}$, there is a constant $C_3>0$ such that $\sup_{j\ge0}v_j\le C_3$ on $\overline{Q_T}$.
Since $v_j(t_j,x_j)\to u^*(\hat{t},\hat{x})$ as $j\to\infty$, for a constant $C_4>0$ (independent of $j$), $v_j(t_j,x_j)\ge u^*(\hat{t},\hat{x})-C_4$ for large $j$. 
Thus, if we set $C_2:=|u^*(\hat{t},\hat{x})-C_4-C_3|$, then
\begin{align*}
(v_j(t_j,x_j)-v_j(t_j-\cdot,x_j))\mathds{1}_{[t_j-a,t_j]}(\cdot)
&\ge (u^*(\hat{t},\hat{x})-C_4-C_3)\mathds{1}_{[t_j-a,t_j]}(\cdot)\\
&\ge -C_2\mathds{1}_{[r,T]}(\cdot)
\end{align*}
on $[0,T]$ for sufficiently large $j$, which is the desired fact.
Note that $-C_2\mathds{1}_{[r,T]}(\tau)/\tau^{\alpha+1}$ is integrable on $[0,T]$. 
Proposition \ref{p:semicontinuityofproduct} also implies that
\begin{align*}
&\liminf_{j\to\infty}(v_j(t_j,x_j)-v_j(t_j-\cdot,x_j))\mathds{1}_{[t_j-a,t_j]}(\cdot)\\
&\ge\liminf_{j\to\infty}(v_j(t_j,x_j)-u^*(t_j-\cdot,x_j))\mathds{1}_{[t_j-a,t_j]}(\cdot)\\
&\ge(u^*(\hat{t},\hat{x})-u^*(\hat{t}-\cdot,\hat{x}))\mathds{1}_{[\hat{t}-a,\hat{t}]}(\cdot).
\end{align*}
Therefore Fatou's lemma can be applied and consequently
$$
\liminf_{j\to\infty}K_{t_j-a}[v_j](t_j,x_j)\ge K_{\hat{t}-a}[u^*](\hat{t},\hat{x})
$$
by combining with \eqref{e:nonintegration}.

Taking the limit inferior $j\to\infty$ to both sides in \eqref{e:jequation} yields \eqref{e:closedgoal}.
\end{proof}

\begin{theorem}[Existence]
\label{t:perron}
Assume (A1).
Let $u^-\in USC(\overline{Q_T})$ and $u^+\in LSC(\overline{Q_T})$ be a supersolution and a subsolution of \eqref{e:fhj} such that $(u^-)_*>-\infty$ and $(u^+)^*<+\infty$ on $\overline{Q_T}$.
Suppose that $u^-\le u^+$ in $\overline{Q_T}$.
Then there exists a solution $u$ of \eqref{e:fhj} that satisfies $u^-\le u\le u^+$ in $\overline{Q_T}$.
\end{theorem}

\begin{proof}
Define
\begin{equation}
\label{e:definesolution}
u(t,x):=\sup\{v(t,x) \mid v\in X\}
\end{equation}
for $(t,x)\in\overline{Q_T}$, where
$$
X:=\{v\in S^- \mid \text{$v\le u^+$ on $\overline{Q_T}$}\}.
$$
Note that $X\neq\emptyset$ since $u^-\in X$.
Also, since $u^-\le u\le u^+$ on $\overline{Q_T}$ by the definition of $u$, $
-\infty<(u^-)_*\le u_*\le u^*\le (u^+)^*<+\infty$ on $\overline{Q_T}$.
Our goal in this proof is to show that $u$ defined by \eqref{e:definesolution} is actually a solution of \eqref{e:fhj}. 
Since we know that $u^*$ is a subsolution of \eqref{e:fhj} from Lemma \ref{l:closedness}, so it suffices to show that $u_*$ is a supersolution of \eqref{e:fhj}.

Suppose by contradiction that $u_*$ were not a supersolution of \eqref{e:fhj}.
Then there would be a function $\phi\in \mathcal{C}^1([0,T]\times\mathbb{R}^d)$, a point $(\hat{t},\hat{x})\in(0,T]\times\mathbb{R}^d$ and a constant $\theta>0$ such that $u_*-\phi$ attains a minimum at $(\hat{t},\hat{x})\in (0,T]\times\mathbb{R}^d$ over $(0,T]\times\mathbb{R}^d$ and
$$
K_0[u_*](\hat{t},\hat{x})+H(\hat{t},\hat{x},u_*(\hat{t},\hat{x}),D\phi(\hat{t},\hat{x}))<-2\theta.
$$
Notice that $K_0[u_*](\hat{t},\hat{x})$ may be $-\infty$, while $K_0[u_*](\hat{t},\hat{x})$ makes sense and is bounded from above by Proposition \ref{p:propertyintegral}.

Let $\rho>0$ be a small parameter so that $\rho<\hat{t}$ and $\overline{B_{2\rho}(\hat{x})}\subset(\hat{x}-\frac{1}{2},\hat{x}+\frac{1}{2}]^d$.
Define functions $w:(0,T]\times\mathbb{R}^d\to\mathbb{R}$ and $U:[0,T]\times (\hat{x}-\frac{1}{2},\hat{x}+\frac{1}{2}]^d\to\mathbb{R}$ by
$$
w(s,y):=\phi(s,y)+\frac{\rho^{2}}{2}-|s-\hat{t}|^2-|y-\hat{x}|^{2}
$$
and
\begin{equation*}
	U=\begin{cases}
		\max\{u^*,w\}\quad&\text{in $((\hat{t}-\rho,\hat{t}+\rho)\times B_{2\rho}(\hat{x}))\cap([0,T]\times (\hat{x}-\frac{1}{2},\hat{x}+\frac{1}{2}]^d)$,}\\
		u^*\quad&\text{in $([0,T]\times (\hat{x}-\frac{1}{2},\hat{x}+\frac{1}{2}]^d)\setminus((\hat{t}-\rho,\hat{t}+\rho)\times B_{2\rho}(\hat{x}))$,}
	\end{cases}
\end{equation*} 
respectively.
We regard $B_{2\rho}(\hat{x})$ and, for each $s\in[0,T]$, $U(s,\cdot)$ as a open ball in $\mathbb{T}^d$ and a function on $\mathbb{T}^d$ by extending it periodically, respectively.
We shall show that $U\in X$ and that there exists a point $(s,y)\in\overline{Q_T}$ such that $U(s,y)>u(s,y)$.
Once these were proved, we would obtain a contradiction due to the maximality of $u$.

Set
$$
\Omega:=\left\{(s,y)\in (\hat{t}-\rho,\hat{t}+\rho)\times B_{2\rho}(\hat{x}) \mid |s-\hat{t}|^2+|y-\hat{x}|^2\le\frac{\rho^2}{2}\right\}\subset \overline{Q_T}.
$$
Then $\overline{\Omega}\subset(\hat{t}-\rho,\hat{t}+\rho)\times B_{2\rho}(\hat{x})$ and 
\begin{equation}
\label{e:observation}
u^*(s,y)\ge u_*(s,y)\ge \phi(s,y)=w(s,y)-\frac{\rho^{2}}{2}+|s-\hat{t}|^{2}+|y-\hat{x}|^{2}> w(s,y)
\end{equation}
for all $(s,y)\in ((\hat{t}-\rho,\hat{t}+\rho)\times B_{2\rho}(\hat{x}))\setminus\Omega$.
Thus $U$ is upper semicontinuous on $\overline{Q_T}$ by its definition.

Assume that $U-\psi$ attains a maximum at $(\hat{s},\hat{y})\in(0,T]\times\mathbb{R}^d$ over $(0,T]\times\mathbb{R}^d$ for $\psi\in \mathcal{C}^1([0,T]\times\mathbb{R}^d)$.
We may assume that $(U-\psi)(\hat{s},\hat{y})=0$.

\textbf{Case 1:} Suppose that $U(\hat{s},\hat{y})=u^*(\hat{s},\hat{y})$.
Then, since $U\ge u^*$ on $\overline{Q_T}$, it turns out that $u^*-\psi$ attains a maximum at $(\hat{s},\hat{y})$ over $(0,T]\times\mathbb{R}^d$ and that
\begin{equation}
\label{e:forU}
U(\hat{s},\hat{y})-U(\hat{s}-\tau,\hat{y})\le u^*(\hat{s},\hat{y})-u^*(\hat{s}-\tau,\hat{y})
\end{equation}
for all $\tau\in[0,\hat{s}]$.
Recall that $u^*$ is a subsolution of \eqref{e:fhj}, so that
$$
K_0[u^*](\hat{s},\hat{y})+H(\hat{s},\hat{y},u(\hat{s},\hat{y}),D\psi(\hat{s},\hat{y}))\le0.
$$
Proposition \ref{p:propertyintegral} (iii) with \eqref{e:forU} ensures that $K_0[U](\hat{s},\hat{y})$ exists and simultaneously $K_0[U](\hat{s},\hat{y})\le K_0[u^*](\hat{s},\hat{y})$.
This implies that $U$ is a subsolution of \eqref{e:fhj}.

\textbf{Case 2:} Suppose that $U(\hat{s},\hat{y})=w(\hat{s},\hat{y})>u(\hat{s},\hat{y})$.
Then, from \eqref{e:observation}, we see $(\hat{s},\hat{y})\in\Omega$, which yields $\lim_{\rho\to0}(\hat{s},\hat{y})=(\hat{t},\hat{x})$.
By employing the idea in \cite[Theorem 3]{i} for example, we shall show that
\begin{equation}
\label{e:claim}
\limsup_{\rho\to0}K_0[U](\hat{s},\hat{y})\le K_0[u_*](\hat{t},\hat{x}).
\end{equation}
Since $U\ge u^*\ge u_*$ on $\overline{Q_T}$ the non-integration term is estimated as
$$
\frac{U(\hat{s},\hat{y})-U(0,\hat{y})}{\hat{s}^{\alpha}\Gamma(1-\alpha)}\le\frac{w(\hat{s},\hat{y})-u_*(0,\hat{y})}{\hat{s}^{\alpha}\Gamma(1-\alpha)}.
$$
Recalling that $\lim_{\rho\to0}w(\hat{s},\hat{y})=\phi(\hat{t},\hat{x})=u_*(\hat{t},\hat{x})$ we see
$$
\limsup_{\rho\to0}\frac{U(\hat{s},\hat{y})-U(0,\hat{y})}{\hat{s}^{\alpha}\Gamma(1-\alpha)}\le\frac{u_*(\hat{t},\hat{x})-u_*(0,\hat{x})}{\hat{t}^{\alpha}\Gamma(1-\alpha)}.
$$
To handle the term of integration let us divide the term of integration in $K_0[U](\hat{s},\hat{y})$ multiplied by $\Gamma(1-\alpha)/\alpha$ into two integrations as follows:
$$
I_{1,\rho}[U]:=\int_0^{\rho^2}(U(\hat{s},\hat{y})-U(\hat{s}-\tau,\hat{y}))\frac{d\tau}{\tau^{\alpha+1}}
$$
and
$$
I_{2,\rho}[U]:=\int_{\rho^2}^{\hat{s}}(U(\hat{s},\hat{y})-U(\hat{s}-\tau,\hat{y}))\frac{d\tau}{\tau^{\alpha+1}}.
$$
By definitions of $U$ and $w$ we have
\begin{align*}
U(\hat{s},\hat{y})-U(\hat{s}-\tau,\hat{y})&\le w(\hat{s},\hat{y})-w(\hat{s}-\tau,\hat{y})\\
&=\phi(\hat{s},\hat{y})-\phi(\hat{s}-\tau,\hat{y})+\tau^2-2(\hat{s}-\hat{y})\tau
\end{align*}
for all $\tau\in[0,\rho^2]$.
Hence $I_{1,\rho}[U]\le I_{1,\rho}[\phi]+C_{\rho}$ with a constant $C_{\rho}$ such that $\lim_{\rho\to0}C_{\rho}=0$.
By Proposition \ref{p:propertyintegral} (iv), we see that $\lim_{\rho\to0}I_{1,\rho}[\phi]=0$, so that
$$
\limsup_{\rho\to0}I_{1,\rho}[U]\le0.
$$
Since $U\ge u_*$ on $\overline{Q_T}$,
\begin{equation}
\label{e:needestimate}
\begin{aligned}
U(\hat{s},\hat{y})-U(\hat{s}-\tau,\hat{y})&\le w(\hat{s},\hat{y})-u_*(\hat{s}-\tau,\hat{y})\\
&\le\phi(\hat{s},\hat{y})-u_*(\hat{s}-\tau,\hat{y})+\frac{\rho^2}{2}
\end{aligned}
\end{equation}
on $[\rho^2,\hat{s}]$.
Moreover, the relationship between $u_*$ and $\phi$ yields
$$
\phi(\hat{s},\hat{y})-u_*(\hat{s}-\tau,\hat{y})+\frac{\rho^2}{2}\le \phi(\hat{s},\hat{y})-\phi(\hat{s}-\tau,\hat{y})+\frac{\rho^2}{2}.
$$
Since $\phi(\cdot,x)$ is continuous on $[0,T]$, we are able to find a large constant $C_1>0$ such that
$$
\phi(\hat{s},\hat{y})-\phi(\hat{s}-\tau,\hat{y})\le C_1\tau
$$
for all $\tau\in[\rho^2,\hat{s}]$.
In addition we may assume that $C_1$ does not depend on $\rho$.
Notice that there exists a constant $C_2>0$ such that $C_1\tau^2+\rho^2/2\le C_2\tau$ for all $\tau\in[\rho^2,\hat{s}]$.
Consequently \eqref{e:needestimate} is lead to
$$
U(\hat{s},\hat{y})-U(\hat{s}-\tau,\hat{y})\le C_2\tau
$$
on $[\rho^2,\hat{s}]$.
The right-hand side with $\tau^{-\alpha-1}$ is integrable on $[0,T]$, so that Fatou's lemma yields
$$
\limsup_{\rho\to0}I_{2,\rho}[U](\hat{s},\hat{y})\le I_{2,0}[u_*](\hat{t},\hat{x}).
$$
The above ensures \eqref{e:claim} and thus we see
$$
K_0[U](\hat{s},\hat{y})-K_0[u_*](\hat{t},\hat{x})\le\theta
$$
for sufficiently small $\rho$.
Notice that this means $K_0[U]$ and $K_0[u_*]$ actually exist. 

Since the maximizer $(\hat{s},\hat{y})$ of $U-\psi$ is of $w-\psi$ (on $\Omega$) as well, the classical maximum principle for $w-\psi$ implies that $D\phi(\hat{s},\hat{y})-2(\hat{y}-\hat{x})=D\psi(\hat{s},\hat{y})$.
Hence we see $\lim_{\rho\to0}D\psi(\hat{s},\hat{y})=D\phi(\hat{t},\hat{x})$.
Moreover, $\lim_{\rho\to0}U(\hat{s},\hat{y})=\lim_{\rho\to0}w(\hat{s},\hat{y})=\phi(\hat{t},\hat{x})=u_*(\hat{t},\hat{x})$.
Therefore
$$
H(\hat{s},\hat{y},U(\hat{s},\hat{y}),D\psi(\hat{s},\hat{y}))-H(\hat{t},\hat{x},u_*(\hat{t},\hat{s}),D\phi(\hat{t},\hat{x}))\le\theta
$$
if $\rho$ is sufficiently small.
Summing up the above we obtain for sufficiently small $\rho$ that
\begin{align*}
&K_0[U](\hat{s},\hat{y})+H(\hat{s},\hat{y},U(\hat{s},\hat{y}),D\psi(\hat{s},\hat{y}))\\
&\le -2\theta+K_0[U](\hat{s},\hat{y})-K_0[u_*](\hat{t},\hat{x})\\
&\quad+H(\hat{s},\hat{y},U(\hat{s},\hat{y}),D\psi(\hat{s},\hat{y}))-H(\hat{t},\hat{x},u_*(\hat{t},\hat{s}),D\phi(\hat{t},\hat{x}))\le0,
\end{align*}
which shows that $U$ is a subsolution of \eqref{e:fhj}.

Theorem \ref{t:comp} implies that $U\le u^+$.
Let $\{(t_j,x_j)\}_{j\ge0}$ be a sequence such that $(t_j,x_j,u(t_j,x_j))\to(\hat{t},\hat{x},u_*(\hat{t},\hat{x}))$ as $j\to\infty$.
Then
\begin{align*}
\liminf_{j\to\infty}(U(t_j,x_j)-u(t_j,x_j))\ge \lim_{j\to\infty}(w(t_j,x_j)-u(t_j,x_j))=\frac{\rho^{2}}{2}>0.
\end{align*}
In other words there exists a point $(s,y)$ such that $U(s,y)>u(s,y)$.
Therefore the proof is complete.
\end{proof}

\begin{corollary}[Unique existence for \eqref{e:fhj}-\eqref{e:init}]
\label{cor:constructbarrier}
Assume (A1)-(A4).
Then there exists at most one solution $u$ of \eqref{e:fhj}-\eqref{e:init}.
\end{corollary}

\begin{proof}
The uniqueness of a solution is guaranteed by Theorem \ref{t:comp}.
Henceforth, it is enough to construct $u^-$ and $u^+$ in Theorem \ref{t:perron} so that $u$ defined by \eqref{e:definesolution} satisfies $u(0,\cdot)=u_0$ on $\mathbb{T}^d$.

Set $\omega(\ell):=\sup\{|u_{0}(\zeta)-u_{0}(\eta)| \mid \zeta,\eta\in\mathbb{T}^d,|\zeta-\eta|\le \ell\}$ for $\ell\ge0$ and $f_y(x):=\sum_{i=1}^d(1-\cos(2\pi(x_i-y_i)))$ for $x,y\in\mathbb{T}^d$, where $x_i$ and $y_i$ are $i$-th components of each variable.
Then for each $\varepsilon>0$ there exists a constant $C_{\varepsilon}>0$ such that $\omega(|x-y|)\le\varepsilon+C_{\varepsilon}f_y(x)$ for all $x,y\in\mathbb{T}^d$.
For $\varepsilon\in(0,1)$ and $y\in\mathbb{T}^{d}$ we define a function $u_{\varepsilon,y}^-:\overline{Q_T}\to\mathbb{R}$ by
$$
u_{\varepsilon,y}^{-}(t,x):=u_{0}(y)-\varepsilon-C_{\varepsilon}f_y(x)-\frac{C t^{\alpha}}{\Gamma(1+\alpha)},
$$
where $C>0$ is a large constant.
Then $u_{\varepsilon,y}^-\in \mathcal{C}^1(\overline{Q_T})$.
Moreover $u_{\varepsilon,y}^-\le u_0(y)$ by the non-negativity of $f_y$ and $|Du_{\varepsilon,y}^-|$ is bounded on $Q_T$.
It is well-known that 
\begin{equation}
\label{e:derivativeofpowerfunction}
\frac{1}{\Gamma(1-\alpha)}\int_a^t\frac{[(s-a)^{\beta}]'}{(t-s)^{\alpha}}ds=\frac{\Gamma(\beta+1)}{\Gamma(\beta-\alpha+1)}(t-a)^{\beta-\alpha}
\end{equation}
for given constants $a\in\mathbb{R}$ and $\beta\in(0,1)$; see \cite[(2.56)]{p} for the proof.
From this formula with $(a,\beta)=(0,\alpha)$ and the above, we see
$$
-C+H(t,x,u_{\varepsilon,y}^-(t,x),Du_{\varepsilon,y}^-(t,x))\le0
$$ 
for all $(t,x)\in Q_T$. 
Thus Proposition \ref{p:consistency} implies that $u_{\varepsilon,y}^-$ is a (viscosity) subsolution of \eqref{e:fhj}.

We also see
$$
u_{\varepsilon,y}^-(t,x)\le u_0(x)+\omega(|x-y|)-\varepsilon-C_{\varepsilon}f_y(x)-\frac{Ct^{\alpha}}{\Gamma(1-\alpha)}\le u_0(x)
$$
for all $(t,x)\in\overline{Q_T}$.
Therefore, Lemma \ref{l:closedness} ensures that 
$$
u^-(t,x):=(\sup\{u_{\varepsilon,y}^-(t,x) \mid \varepsilon\in(0,1),y\in\mathbb{T}^d\})^*
$$
 is a subsolution of \eqref{e:fhj} and satisfies $u^-(t,x)\le u_0(x)$ for all $(t,x)\in\overline{Q_T}$.
The definition of $u^-$ yields $u^-(0,\cdot)\ge u_0$ on $\mathbb{T}^d$, which guarantees that $(u^-)_*>-\infty$ on $\overline{Q_T}$ and $u^-(0,\cdot)=u_0$ on $\mathbb{T}^d$.
Similarly, a supersolution with desired properties is constructed.
Moreover, it turns out that $u^\pm$ satisfy $u_0(x)=\lim_{(t,y)\to(x,0)}u^\pm(t,y)$ but we leave the verification to the reader; cf. \cite{g}.
With $u^\pm$ above, we obtain a solution $u$ by Theorem \ref{t:perron}, and it satisfies $u(0,\cdot)=u_0$ on $\mathbb{T}^d$.
The proof is now complete.
\end{proof}

\section{Some stability results}

Two main theorems in this section are in what follows:
\begin{theorem}[Stability I]
\label{t:stab}
Let $H_\varepsilon$ and $H$ satisfy (A1)-(A3), where $\varepsilon>0$.
Let $u_\varepsilon\in USC(\overline{Q_T})$ (resp. $LSC(\overline{Q_T})$) be a subsolution (resp. supersolution) of 
$$
\partial_t^\alpha u_\varepsilon+H_\varepsilon(t,x,u_\varepsilon,Du_\varepsilon)=0\quad\text{in $Q_T$.}
$$
Assume that $H_\varepsilon$ converges to $H$ as $\varepsilon\to0$ locally uniformly in $(0,T]\times\mathbb{T}^d\times\mathbb{R}\times\mathbb{R}^d$.
Assume that $\{u_{\varepsilon}\}_{\varepsilon>0}$ is locally uniformly bounded.
Then $u:=\limsup{}^{*}u_{\varepsilon}$ (resp. $\liminf{}_{*}u_{\varepsilon}$) is a subsolution (resp. supersolution) of 
$$
\partial_t^\alpha u+H(t,x,u,Du)=0\quad\text{in $Q_T$.}
$$
\end{theorem}

Here $\limsup{}^{*}u_{\varepsilon}$ appears above is the \emph{upper relaxed limit} defined by
\begin{equation}
\label{e:url}
(\limsup_{\varepsilon\to0}{}^{*}u_{\varepsilon})(t,x):=\lim_{\delta\searrow0}\sup\{u_{\varepsilon}(s,y) \mid (s,y)\in Q_T\cap\overline{B_{\delta}(t,x)},0<\varepsilon<\delta\}
\end{equation}
for $(t,x)\in\overline{Q_T}$ and $\liminf{}_{*}u_{\varepsilon}:=-\limsup{}^{*}(-u_{\varepsilon})$ is the \emph{lower relaxed limit}.

\begin{theorem}[Stability I\hspace{-.1em}I]
\label{t:conv}
Assume that (A1)-(A4).
Let $u_{\alpha}\in C(\overline{Q_T})$ be a solution of \eqref{e:fhj}-\eqref{e:init} whose time-derivative's order is $\alpha\in(0,1)$.
Then $u_{\alpha}$ converges to $u_{\beta}$ locally uniformly in $Q_T$ as $\alpha\to\beta$, where $u_{\beta}$ is a solution of \eqref{e:fhj}-\eqref{e:init} whose time-derivative's order is $\beta\in(0,1]$. 
\end{theorem}

A same idea as for the proof of \cite[Theorem V.1.7]{bcd} is used for Theorem \ref{t:stab} and Theorem \ref{t:conv}.
A deal for the term of time-derivative is the only difference between Theorem \ref{t:stab} and \cite[Theorem V.1.7]{bcd} but it is similar between Theorem \ref{t:stab} and Theorem \ref{t:conv}.
For this reason we only prove Theorem \ref{t:conv}.

\begin{proof}
As the analogy of the upper/lower relaxed limits, for $\beta\in(0,1]$, we define functions $u^\sharp$ and $u_\sharp$ by
$$
u^\sharp(t,x):=\lim_{\delta\searrow0}\sup\{u_{\alpha}(s,y) \mid (s,y)\in\overline{B_{\delta}(t,x)}\cap Q_T,\alpha\in(\beta-\delta,\beta+\delta)\cap(0,1)\}
$$
and $u_\sharp:=-(-u)^\sharp$ on $\overline{Q_T}$.
In Remark \ref{r:unifbdd} later we will mention that $\{u_{\alpha}\}_{\alpha\in(0,1]}$ of \eqref{e:fhj}-\eqref{e:init} is uniformly bounded on $\overline{Q_T}$.
Hence $u^\sharp$ and $u_\sharp$ are bounded on $\overline{Q_T}$.
Note also that $u^\sharp$ is an upper semicontinuous function, so $u_\sharp$ is a lower semicontinuous function. 

We shall show that $u^\sharp$ and $u_\sharp$ are a subsolution and a supersolution of \eqref{e:fhj}-\eqref{e:init} whose time-derivative's order is $\beta$.
It suffices to show that $u^\sharp$ is a subsolution of \eqref{e:fhj} since the similar argument is applied for $u_\sharp$ and it is clear that $u^\sharp(0,\cdot)\le u_0$ and $u_\sharp(0,\cdot)\ge u_0$ on $\mathbb{T}^{d}$.

Fix $[a,b]\times B\subset(0,T]\times\mathbb{R}^d$ arbitrarily, where $a<b$ and $B$ is an open set in $\mathbb{R}^d$.
Assume that $u^\sharp-\phi$ attains a maximum at $(\hat{t},\hat{x})\in(a,b]\times B$ over $[a,b]\times\overline{B}$ for $\phi\in\mathcal{C}^{1}([0,T]\times\mathbb{R}^d)$.
Let $\{\alpha_j\}_{j\ge0}$ and $\{(t_j,x_j)\}_{j\ge0}$ be sequences such that $u_{\alpha_j}-\phi$ attains a maximum at $(t_j,x_j)\in(a,b]\times B$ over $[a,b]\times \overline{B}$ and 
$$
(\alpha_j,t_j,x_j,u_{\alpha_j}(t_j,x_j))\to(\beta,\hat{t},\hat{x},u^\sharp(\hat{t},\hat{x}))
$$
as $j\to\infty$.
A proof of existence of such sequences is essentially same as for \cite[Lemma V.1.6]{bcd} and not difficult, so we omit it.  

\textbf{Case 1: $\beta\neq1$.} 
Since $u_{\alpha_j}$ is a subsolution of \eqref{e:fhj},
\begin{equation}
\label{e:unlimited}
J_{t_j-a}^{\alpha_j}[\psi_j](t_j,x_j)+K^{\alpha_j}_{t_j-a}[u_{\alpha_j}](t_j,x_j)+H(t_j,x_j,u_{\alpha_j}(t_j,x_j),D\phi(t_j,x_j))\le0.
\end{equation}
Here $J_r^{\alpha_j}$ and $K_r^{\alpha_j}$ are associated with $\alpha=\alpha_j$.
By similar arguments in previous sections it can be turns out that 
\begin{equation}
\label{e:limit}
\liminf_{j\to\infty}(J_{t_j-a}^{\alpha_j}[\psi_j](t_j,x_j)+K^{\alpha_j}_{t_j-a}[u_{\alpha_j}](t_j,x_j))\ge J_{\hat{t}-a}^{\beta}[\phi](\hat{t},\hat{x})+K_{\hat{t}-a}^{\beta}[u^\sharp](\hat{t},\hat{x}).
\end{equation}
Since
$$
\lim_{j\to\infty}H(t_j,x_j,u_{\alpha_j}(t_j,x_j),D\phi(t_j,x_j))=H(\hat{t},\hat{x},u^{\sharp}(\hat{t},\hat{x}),D\phi(\hat{t},\hat{x})),
$$
we find that $u^\sharp$ is a subsolution of \eqref{e:fhj}.

\textbf{Case 2:} $\beta=1$. 
There are similar sequences $\{\alpha_j\}_j$ and $\{(t_j,x_j)\}_j$ even for $\varphi\in C^1((0,T]\times\mathbb{R}^d)$ instead of $\phi\in\mathcal{C}^1([0,T]\times\mathbb{R}^d)$.
Since $(t_j,x_j)\to(\hat{t},\hat{x})\in(a,b]\times B$ as $j\to\infty$, we may assume that $\{(t_j,x_j)\}_j\subset (\hat{t}-3\delta,\hat{t}]\times B_{2\delta}(\hat{x})$ by considering large $j$, where $\delta>0$ is a constant such that $[\hat{t}-3\delta,\hat{t}]\times \overline{B_{2\delta}(\hat{x})}\subset (a,b]\times B$. 
Let $\xi_1,\xi_2:[0,T]\times\mathbb{R}^d\to\mathbb{R}$ be $C^\infty$ functions such that $\xi_1+\xi_2=1$ on $[0,T]\times\mathbb{R}^d$, $\xi_1=1$ on $[\hat{t}-\delta,\hat{t}]\times B_\delta(\hat{x})$ and $\xi_2=1$ on $([0,T]\times\mathbb{R}^d)\setminus([\hat{t}-2\delta,\hat{t}]\times B_{2\delta}(\hat{x}))$.
Since $\{u_\alpha\}_{\alpha\in(0,1]}$ is uniformly bounded on $\overline{Q_T}$, there is a constant $C>0$ such that $\max_{\overline{Q_T}}|u_{\alpha_j}|\le C$ for all $j$.
Set $\psi=\xi_1\phi+\xi_2 M$, where $M:=C+1$.
Then $\psi\in C^1([0,T]\times\mathbb{R}^d)\subset\mathcal{C}^1([0,T]\times\mathbb{R}^d)$ and $u_{\alpha_j}-\psi$ attains a maximum at $(t_j,x_j)\in (0,T]\times\mathbb{R}^d$.
Thus we have 
\begin{equation}
\label{e:usual}
K_0[u_{\alpha_j}](t_j,x_j)+H(t_j,x_j,u_{\alpha_j}(t_j,x_j),D\psi(t_j,x_j))\le0.
\end{equation}
Note that $K_0[u_{\alpha_j}](t_j,x_j)\ge \partial_t^{\alpha_j}\psi(t_j,x_j)$.
According to \cite[Theorem 2.10]{d} we notice that
$$
\lim_{j\to\infty}\partial_{t}^{\alpha_j}\psi(t_j,x_j)=\partial_{t}\psi(\hat{t},\hat{x}).
$$
Thus estimating $K_0[u_{\alpha_j}](t_j,x_j)$ by $\partial_t^{\alpha_j}\psi(t_j,x_j)$ in \eqref{e:usual} and then passing to the limit $j\to\infty$ implies that
$$
\partial_{t}\psi(\hat{t},\hat{x})+H(\hat{t},\hat{x},u^\sharp(\hat{t},\hat{x}),D\psi(\hat{t},\hat{x}))\le0.
$$
Since $\partial_t\psi(\hat{t},\hat{x})=\partial_t\phi(\hat{t},\hat{x})$ and $D\psi(\hat{t},\hat{x})=D\phi(\hat{t},\hat{x})$, $u^\sharp$ is a subsolution of \eqref{e:fhj} with $\alpha=1$.

The comparison principle implies that $u^\sharp\le u_\sharp$ on $\overline{Q_T}$ but $u_\sharp\le u^\sharp$ by their definition.
We hence see that $u:=u^\sharp=u_\sharp$ is a solution of \eqref{e:fhj} and $u(0,\cdot)=u_0$ on $\mathbb{T}^{d}$.
Corollary \ref{c:unique} ensures that $u=u_{\beta}$, a conclusion.
\end{proof}

\section{Regularity results}
\label{s:regularity}
Let consider one-dimensional transport equations of the form
$$
\partial_t^\alpha u+\partial_xu=0\quad\text{in $(0,\infty)\times\mathbb{R}$}
$$
with prescribed initial value $u|_{t=0}=u_0\in C(\mathbb{R})$.
In \cite{mmp} for instance, a solution of this equation was given as
\begin{equation}
\label{e:representation}
u(t,x)=\frac{1}{t^{\alpha}}\int_0^{\infty}W_{-\alpha,1-\alpha}\left(-\frac{z}{t^{\alpha}}\right)u_0(x-z)dz
\end{equation}
through the Laplace and the inverse Laplace transformation.
Here $W_{-\alpha,1-\alpha}$ is Wright function defined by
$$
W_{-\alpha,1-\alpha}(z):=\sum_{j=0}^{\infty}\frac{z^j}{j!\Gamma(-\alpha j+1-\alpha)}.
$$
For properties and formulae for Wright function, see \cite{p} and references therein.
It can be verified that this solution is indeed a unique viscosity solution.
We leave the detail of calculations to the reader.   

Let us assume that
\begin{itemize}
\item[(A4')] $u_0$ is a Lipschitz continuous function with the Lipschitz constant $\lip[u_0]$.
\end{itemize}
We shall prove a continuity of solutions with respect to each variable.
In space, since $W_{-\alpha,1-\alpha}\ge0$ on $(0,\infty)$ and
$$
\int_0^{\infty}W_{-\alpha,1-\alpha}(-z)dz=1,
$$ 
we have 
\begin{align*}
|u(t,x)-u(t,y)|&\le\frac{1}{t^{\alpha}}\int_0^{\infty}W_{-\alpha,1-\alpha}\left(-\frac{z}{t^{\alpha}}\right)|u_0(x-z)-u_0(y-z)|dz\\
&\le \frac{1}{t^{\alpha}}\int_0^{\infty}W_{-\alpha,1-\alpha}\left(-\frac{z}{t^{\alpha}}\right)dz\lip[u_0]|x-y|=\lip[u_0]|x-y|
\end{align*}
for $(t,x,y)\in[0,T]\times\mathbb{R}\times\mathbb{R}$.
In time, since
$$
\int_0^{\infty}W_{-\alpha,1-\alpha}(-z)zdz=\frac{1}{\Gamma(\alpha+1)}
$$
and $u$ given by \eqref{e:representation} is rewritten as 
$$
u(t,x)=\int_0^{\infty}W_{-\alpha,1-\alpha}(-z)u_0(x-t^{\alpha}z)dz,
$$
we have
\begin{align*}
|u(t,x)-u(s,x)|&\le\int_0^{\infty}W_{-\alpha,1-\alpha}(-z)|u_0(x-t^{\alpha}z)-u_0(x-s^{\alpha}z)|dz\\
&\le \int_{0}^{\infty}W_{-\alpha,1-\alpha}(-z)zdz\lip[u_0]|t^{\alpha}-s^{\alpha}|\le \frac{\lip[u_0]}{\Gamma(\alpha+1)}|t-s|^{\alpha}
\end{align*}
for $(t,s,x)\in[0,T]\times[0,T]\times\mathbb{R}$.
These regularity results hold for solutions of \eqref{e:fhj}-\eqref{e:init} with some general Hamiltonians $H$ under some Lipschitz continuity in $x$ for $H$ as for $\alpha=1$.

\begin{lemma}[Lipschitz preserving]
\label{l:lip}
Assume (A1), (A3), (A4') and that there exist constants $L_1\ge0$ and $L_2>0$ such that
\begin{equation*}
|H(t,x,r,p)-H(t,y,r,p)|\le L_1|x-y|+L_2|x-y||p|
\end{equation*}
for all $(t,x,y,r,p)\in(0,T]\times\mathbb{R}^d\times\mathbb{R}^d\times\mathbb{R}\times\mathbb{R}^d$.
Let $u\in C(\overline{Q_T})$ be a solution of \eqref{e:fhj}-\eqref{e:init}.
Then $|u(t,x)-u(t,y)|\le L(t)|x-y|$ for all $(t,x,y)\in[0,T]\times\mathbb{R}^d\times\mathbb{R}^d$ with
$$
L(t)=\left(\lip[u_0]+\frac{L_1}{L_2}\right)E_{\alpha}(L_2 t^{\alpha})-\frac{L_1}{L_2}.
$$
\end{lemma}

\begin{proof}
For $\delta$ we set
$$
\Phi_{\delta}(t,x,y):=u(t,x)-u(t,y)-L_{\delta}(t)|x-y|,
$$
where
$$
L_{\delta}(t)=\left(\lip[u_0]+\frac{L_1+\delta}{L_2}\right)E_{\alpha}(L_2 t^{\alpha})-\frac{L_1+\delta}{L_2}.
$$
Note that $L_{\delta}\in C^{1}((0,T])\cap C([0,T])$ and $L_{\delta}'\in L^{1}(0,T)$.
Suppose by contradiction that there would exit $\delta>0$ such that $\max_{[0,T]\times\mathbb{R}^d\times\mathbb{R}^d}\Phi_{\delta}>0$.

Let $(\hat{t},\hat{x},\hat{y})\in[0,T]\times\mathbb{R}^d\times\mathbb{R}^d$ be a maximum point of $\Phi$.
Note that $\hat{t}>0$ and $\hat{x}\neq\hat{y}$; otherwise $0<\Phi(\hat{t},\hat{x},\hat{x})=0$ or $0<\Phi(0,\hat{x},\hat{y})\le0$ since $u(0,\cdot)=u_0$ is Lipschitz continuous and $E_{\alpha}(0)=1$.
Moreover $u(\hat{t},\hat{x})\ge u(\hat{t},\hat{y})$ from $\Phi_{\delta}(\hat{t},\hat{x},\hat{y})>0$.

Since $u$ is a solution of \eqref{e:fhj}, we have from Lemma \ref{l:ishiilemma}
\begin{equation}
\label{e:lip}
K_0[u](\hat{t},\hat{x})-K_0[u](\hat{t},\hat{y})+H(\hat{t},\hat{x},u(\hat{t},\hat{x}),L_{\delta}(\hat{t})\hat{p})-H(\hat{t},\hat{y},u(\hat{t},\hat{y}),L_{\delta}(\hat{t})\hat{p})\le0,
\end{equation}
where $\hat{p}:=(\hat{x}-\hat{y})/|\hat{x}-\hat{y}|$.
It follows from the definition of $\Phi$ that 
$$
K_0[u](\hat{t},\hat{x})-K_0[u](\hat{t},\hat{y})\ge K_0[L_{\delta}](\hat{t})|\hat{x}-\hat{y}|=(\partial_t^\alpha L_{\delta})(\hat{t})|\hat{x}-\hat{y}|.
$$
Hamiltonians are estimates as
\begin{align*}
&H(\hat{t},\hat{x},u(\hat{t},\hat{x}),L_{\delta}(\hat{t})\hat{p})-H(\hat{t},\hat{y},u(\hat{t},\hat{y}),L_{\delta}(\hat{t})\hat{p})\\
&\ge H(\hat{t},\hat{x},u(\hat{t},\hat{y}),L_{\delta}(\hat{t})\hat{p})-H(\hat{t},\hat{y},u(\hat{t},\hat{y}),L_{\delta}(\hat{t})\hat{p})\\
&\ge -L_1|\hat{x}-\hat{y}|-L_2 L_{\delta}(\hat{t})|\hat{x}-\hat{y}|
\end{align*}
by (A3) and (A2') with $|\hat{p}|=1$.
Thus \eqref{e:lip} is led to
$$
[(\partial_t^{\alpha}L_{\delta})(\hat{t})-L_2L_{\delta}(\hat{t})-L_1]|\hat{x}-\hat{y}|\le0.
$$
Since $L_{\delta}$ satisfies 
$$
\partial_t^\alpha L_{\delta}-L_2L_{\delta}-L_1=\delta
$$
according to \cite[Theorem 7.2, Remark 7.1]{d}, we obtain
$$
\delta|\hat{x}-\hat{y}|\le0,
$$
a contradiction.

Consequently for any $\delta>0$ we see that $|u(t,x)-u(t,y)|\le L_{\delta}|x-y|$ for all $(t,x,y)\in[0,T]\times\mathbb{R}^d\times\mathbb{R}^d$.
Letting $\delta\to0$ yields the conclusion.
\end{proof}

\begin{lemma}
\label{l:cont}
Assume (A1), (A2), (A3) and (A4').
Let $u\in C(\overline{Q_T})$ be a solution of \eqref{e:fhj}-\eqref{e:init}.
Then there exists a constant $M>0$ depending only on $H$, $u_0$, $\alpha$ and $T$ such that
$$
|u(t,x)-u_0(x)|\le M t^{\alpha}
$$
for all $(t,x)\in\overline{Q_T}$.
\end{lemma}

\begin{proof}
We find that $u^-(t,x):=u_0(x)-Mt^{\alpha}$ and $u^+(t,x):=u_0(x)+Mt^{\alpha}$ are a subsolution and a supersolution of \eqref{e:fhj}-\eqref{e:init}, respectively, where the constant $M$ is chosen so large that
$$
M\ge\sup\{\Gamma^{-1}(\alpha+1)|H(t,x,\max_{\mathbb{T}^d}|u_0|,p)| \mid (t,x)\in Q_T,|p|\le \lip[u_0],\alpha\in(0,1)\}.
$$
In fact, if $u^- -\phi$ attains a maximum at $(\hat{t},\hat{x})\in(0,T]\times\mathbb{R}^d$ for $\phi\in\mathcal{C}^1([0,T]\times\mathbb{R}^d)$, then $|D\phi(\hat{t},\hat{x})|\le \lip[u_0]$ and $K_0[u^-](\hat{t},\hat{x})=-M \partial_t^\alpha t^\alpha|_{t=\hat{t}}=\Gamma(1+\alpha)M$ by using the formula $\partial_t^{\alpha}t^{\alpha}=\Gamma(1+\alpha)$ derived from \eqref{e:derivativeofpowerfunction}.
Therefore
$$
K_0[u^-](\hat{t},\hat{x})+H(\hat{t},\hat{x},u^-(\hat{t},\hat{x}),D\phi(\hat{t},\hat{x}))
\le -\Gamma(1+\alpha)M+H(\hat{t},\hat{x},\max_{\mathbb{T}^d}|u_0|,D\phi(\hat{t},\hat{x}))\le0.
$$
by (A3) and the choice of $M$ since $u^-(t,x)=u_0(x)-Mt^\alpha\le\max_{\mathbb{T}^d}|u_0|$.
Similarly, it is verified that $u^+$ is a supersolution of \eqref{e:fhj}.

Theorem \ref{t:comp} (comparison principle) yields to
$$
u_0(x)-Mt^{\alpha}=u^-(t,x)\le u(t,x)\le u^+(t,x)=u_0(x)+Mt^{\alpha}
$$
on $\overline{Q_T}$, which is noting but the desired estimate.
\end{proof}

\begin{remark}
\label{r:unifbdd}
Let $u_{\alpha}$ be a solution of \eqref{e:fhj}-\eqref{e:init} whose time-derivative's order is $\alpha\in(0,1]$.
For $u_0\in C(\mathbb{T}^{d})$ not necessarily Lipschitz continuous, we observe that
\begin{equation}
\label{e:unifbdd}
\sup\{|u_\alpha(t,x)| \mid (t,x)\in\overline{Q_T},\alpha\in(0,1]\}\le\max_{\mathbb{T}^{d}}|u_0|+C\max\{1,T\}. 
\end{equation}
Here $C>0$ is a large constant so that
$$
C\ge\sup\{\Gamma^{-1}(\alpha+1)|H(t,x,\max_{\mathbb{T}^d}|u_0|,0)| \mid (t,x)\in \overline{Q_T},\alpha\in(0,1]\}.
$$
Indeed, $\max_{\mathbb{T}^{d}}|u_0|-Ct^{\alpha}$ and $-\max_{\mathbb{T}^{d}}|u_0|+Ct^{\alpha}$ are a (classical) subsolution and a (classical) supersolution of \eqref{e:fhj}-\eqref{e:init}, respectively.
Thus the comparison principle implies \eqref{e:unifbdd} once one realizes that $t^{\alpha}\le T^{\alpha}\le\max\{1,T\}$ for all $t$, $T$ and $\alpha\in(0,1]$.
\end{remark}

\begin{lemma}[Temporal H\"{o}lder continuity]
\label{l:holder}
Assume (A1), (A2), (A3) and (A4').
Let $u\in C(\overline{Q_T})$ be a solution of \eqref{e:fhj}-\eqref{e:init}.
Then for the same constant $M>0$ as in Lemma \ref{l:cont} 
\begin{equation}
\label{e:holder}
|u(t,x)-u(s,x)|\le M|t-s|^{\alpha}
\end{equation}
for all $(t,s,x)\in[0,T]\times[0,T]\times\mathbb{T}^{d}$.
\end{lemma}

\begin{proof}
Let $X$ be a set of subsolutions $v$ of \eqref{e:fhj}-\eqref{e:init} such that $v$  satisfy \eqref{e:holder} and $u_0-Mt^\alpha\le v\le u_0+Mt^\alpha$ on $\overline{Q_T}$.
Notice that $X\neq\emptyset$ since $u_0(x)-Mt^\alpha\in X$ due to Lemma \ref{l:cont}.
Define $u=\sup\{v\mid v\in X\}$.
We show by Perron's method that $u$ is a solution of \eqref{e:fhj}-\eqref{e:init} satisfying \eqref{e:holder}.
In view of Corollary \ref{cor:constructbarrier} it is enough to prove that $u$ is a solution of \eqref{e:fhj} satisfying \eqref{e:holder}.
In this proof we use same notations associated to the above $u$ as in Theorem \ref{t:perron}.

It is not hard to see that
\begin{equation}
\label{e:similarinequality}
|u(t,x)-u(s,x)|\le \sup\{|v(t,x)-v(s,x)|\mid v\in X\}
\end{equation}
for all $t,s\in[0,T]$ and $x\in\mathbb{T}^d$.
Since $M$ does not depend unknown functions and $v$ satisfies \eqref{e:holder}, we see that $u$ satisfies \eqref{e:holder}.
Let us prove that $u$ is a solution of \eqref{e:fhj}.

To do so we must show that the function $U$ satisfies \eqref{e:holder}.
All processes except for this step work to the current situation.
We first show that the function $w$ satisfies \eqref{e:holder} near $(\hat{t},\hat{x})$.
Expanding $w$ by Taylor formula, we have
\begin{align*}
w(s,y)-w(\hat{t},\hat{x})&=\phi(s,y)-\phi(\hat{t},\hat{x})-|s-\hat{t}|^2-|y-\hat{x}|^2\\
&=a(s-\hat{t})+p\cdot(y-\hat{x})+o(|s-\hat{t}|+|y-\hat{x}|)-|s-\hat{t}|^2-|y-\hat{x}|^2
\end{align*}
for $Q_T\ni(s,y)\to(\hat{t},\hat{x})$, where $a=\partial_t\phi(\hat{t},\hat{x})$ and $p=D\phi(\hat{t},\hat{x})$.
For every $\eta>0$ there exists $\delta>0$ such that
$$
|w(s,y)-w(\hat{t},\hat{x})|\le ((|a|+\eta)|s-\hat{t}|^{1-\alpha}+|s-\hat{t}|^{2-\alpha})|s-\hat{t}|^\alpha+(|p|+\eta)|y-\hat{x}|+|y-\hat{x}|^2
$$
for all $(s,y)\in\overline{B_\delta(\hat{t},\hat{x})}$.
Fix such $\eta>0$.
For $\delta$ so that $(|a|+\eta)\delta^{1-\alpha}+\delta^{2-\alpha}\le M$, $w$ satisfies \eqref{e:holder} on $\overline{B_\delta(\hat{t},\hat{x})}$.

Let $\rho$ be taken so small that $\rho\le \sqrt{2}\delta$.
Then $\Omega\subset B_\delta(\hat{t},\hat{x})$.
Since $u$ and $w$ satisfy \eqref{e:holder} in $((\hat{t}-\rho,\hat{t}+\rho)\times B_{2\rho}(\hat{x}))\cap B_{\delta}(\hat{t},\hat{x})$, it turns out by a similar inequality for the function $\max\{u^*,w\}$ as \eqref{e:similarinequality} that $\max\{u^*,w\}$ satisfies \eqref{e:holder} in the same region.
Since $u^*>w$ in $((\hat{t}-\rho,\hat{t}+\rho)\times B_{2\rho}(\hat{x}))\setminus\Omega$, in consequence, $U$ satisfies \eqref{e:holder} in $\overline{Q_T}$, a conclusion.
\end{proof}

\section{Another possible definition}\label{s:comment}
In this section, for simplicity, we only treat Hamiltonians independent of $t$ and $r$, i.e., $H=H(x,p)$ and assume (A1)-(A2).
Following to the usual style for viscosity solutions we are also able to define weak solutions of \eqref{e:fhj} as follows:

\begin{definition}[Provisional solutions]
\label{d:provisionalsolution}
For a function $u\in C(\overline{Q_T})$ we call $u$ a \emph{provisional subsolution} (resp. \emph{provisional supersolution}) of \eqref{e:fhj} if
$$
(\partial_{t}^{\alpha}\phi)(\hat{t},\hat{x})+H(\hat{x},D\phi(\hat{t},\hat{x}))\le0\quad\text{(resp. $\ge0$)}
$$ 
whenever $u-\phi$ attains a maximum (resp. minimum) at $(\hat{t},\hat{x})\in Q_T$ over $\overline{Q_T}$ for $\phi\in\mathcal{C}^{1}([0,T]\times\mathbb{R}^d)$.
If $u\in C(\overline{Q_T})$ is a both provisional sub- and supersolution of \eqref{e:fhj}, then we call $u$ a \emph{provisional solution} of \eqref{e:fhj}.
\end{definition}

It is no difficulties to prove that provisional solutions of \eqref{e:fhj} are consistent with classical solutions of \eqref{e:fhj} if they belong to $\mathcal{C}^{1}(\overline{Q_T})$; cf. Proposition \ref{p:consistency} and \cite[Theorem 1]{l}.
Although Definition \ref{d:provisionalsolution} looks good, there are some technical difficulties to handle provisional solutions.
We conclude this paper by sharing a main part of such difficulties.

They occurs in a proof of comparison principle.
Let $u$ and $v$ be respectively a provisional subsolution and a provisional supersolution of \eqref{e:fhj} such that $u(0,\cdot)\le v(0,\cdot)$ on $\mathbb{T}^{d}$.
Suppose that $\max_{\overline{Q_T}}(u-v)>0$ and aim to derive a contradiction.

There is a small constant $\eta>0$ such that
$$
\max_{(t,x)\in\overline{Q_T}}((u-v)(t,x)-\eta t^{\alpha})=:\theta>0.
$$
For $\varepsilon>0$ and $\delta>0$ we consider the function
$$
\Phi(t,s,x,y):=u(t,x)-v(s,y)-\frac{|x-y|^{2}}{2\varepsilon}-\frac{|t-s|^{2}}{2\delta}-\eta t^{\alpha}.
$$
on $[0,T]^{2}\times\mathbb{T}^{2d}$.
Let $(\bar{t},\bar{s},\bar{x},\bar{y})$ be a maximum point of $\Phi$.
From inequalities for provisional sub- and supersolutions, we have
\begin{equation}
\label{e:vineq}
\left(\partial_t^{\alpha}\frac{|\cdot-\bar{s}|^{2}}{2\delta}\right)(\bar{t},\bar{s})+\left(\partial_{s}^{\alpha}\frac{|\bar{t}-\cdot|}{2\delta}\right)(\bar{t},\bar{s})+\eta\Gamma(1+\alpha)+H(\bar{x},\bar{p})-H(\bar{y},\bar{p})\le0.
\end{equation}
Here $\bar{p}:=(\bar{x}-\bar{y})/\varepsilon$.
A similar argument is found in Section \ref{s:comp} of this paper.
The third term comes from the last term of $\Phi$, i.e, $\eta t^{\alpha}$.
Let us focus on the first and second terms.
A direct calculation implies that
$$
\partial_{t}^{\alpha}|t-s|^{2}=\frac{2(t-(2-\alpha)s)t^{1-\alpha}}{\Gamma(3-\alpha)}
$$
by the formula \eqref{e:derivativeofpowerfunction}.
By changing the role of $t$ and $s$, consequently, we get 
\begin{equation}
\label{e:deriv}
\begin{aligned}
\left(\partial_t^{\alpha}\frac{|\cdot-\bar{s}|^{2}}{2\delta}\right)(\bar{t})+\left(\partial_{s}^{\alpha}\frac{|\bar{t}-\cdot|^2}{2\delta}\right)(\bar{s})
&=\frac{(\bar{t}-(2-\alpha)\bar{s})\bar{t}^{1-\alpha}+(\bar{s}-(2-\alpha)\bar{t})\bar{s}^{1-\alpha}}{\delta\Gamma(1-\alpha)}\\
&=\frac{\bar{t}^{2-\alpha}+\bar{s}^{2-\alpha}-(2-\alpha)(\bar{s}\bar{t}^{1-\alpha}+\bar{t}\bar{s}^{1-\alpha})}{\delta\Gamma(3-\alpha)}.
\end{aligned}
\end{equation}

When $\alpha=1$, \eqref{e:deriv} vanishes.
Thus estimating Hamiltonians suitably (see the proof of Theorem \ref{t:comp} in this paper) and then passing to the limit $\varepsilon,\delta\to0$ in \eqref{e:vineq} yields the contradiction thanks to the third term.
On the other hand, the situation for $\alpha\in(0,1)$ is completely different.
Indeed, \eqref{e:deriv} possibly does not vanish and it is hard to control $(\bar{t},\bar{s})$ so that \eqref{e:deriv} is sufficiently small comparing to $\eta\Gamma(1+\alpha)$ as $\delta\to0$. 
There is a possibility that \eqref{e:deriv} diverges as $\delta\to0$ as well.

To solve above difficulties let us consider the following problem:
\begin{problem}
\label{problem}
Find a function $\psi\in C^1((0,T]^2;\mathbb{R})\cap C([0,T]^2)$ satisfying $\partial_t\psi(\cdot,s)\in L^1(0,T)$ for every $s\in[0,T]$ and $\partial_s\psi(t,\cdot)\in L^1(0,T)$ for every $t\in[0,T]$ such that
\begin{equation}
\label{e:question}
	\begin{cases}
		\partial_t^{\alpha}\psi(\cdot,s)+\partial_s^{\alpha}\psi(t,\cdot)\ge0\quad&\text{on $(0,T)^{2}$,}\\
		\psi=0\quad&\text{on $\{t=s\}$ and}\\
		\psi>0\quad&\text{on $[0,T]^{2}\setminus\{t=s\}$.} 
	\end{cases}
\end{equation}
\end{problem}
If we could find such a function, then the contradiction would be obtained by handling
$$
\Psi(t,s,x,y):=u(t,x)-v(s,y)-\frac{|x-y|^{2}}{2\varepsilon}-\frac{\psi(t,s)}{2\delta}-\eta t^{\alpha}
$$
instead of $\Phi$.
However, such a modification unfortunately does not overcome the difficulty yet.
\begin{proposition}
There is no function $\psi$ solving Problem \ref{problem}.
\end{proposition}

\begin{proof}
Suppose by contradiction that there is a function $\psi$ solves Problem \ref{problem}.
Then $\psi$ should satisfy
$$
(\partial_t^\alpha\psi)(t,t)+(\partial_s^{\alpha}\psi)(t,t)\ge0,
$$
that is,
\begin{equation}
\label{e:deriv2}
\int_0^t\frac{(\partial_t\psi)(\tau,t)}{(t-\tau)^{\alpha}}d\tau+\int_0^t\frac{(\partial_s\psi)(t,\tau)}{(t-\tau)^{\alpha}}d\tau\ge0.
\end{equation}
Since $\psi(t,t)=0$ and $\psi(\cdot,t)\in C^{1}(0,T)$, integration by parts implies that
$$
\int_0^t\frac{(\partial_t\psi)(\tau,t)}{(t-\tau)^{\alpha}}d\tau=-\frac{\psi(0,t)}{t^{\alpha}}-\alpha\int_0^t\frac{\psi(t,\tau)}{(t-\tau)^{\alpha+1}}d\tau.
$$
Thus \eqref{e:deriv2} is rewritten as
$$
\frac{\psi(0,t)+\psi(t,0)}{t^{\alpha}}+\alpha\int_0^t\frac{\psi(t,\tau)+\psi(\tau,t)}{(t-\tau)^{\alpha+1}}d\tau\le0.
$$
However the left-hand side is positive since $\psi>0$ on $[0,T]^2\setminus\{t=s\}$, a contradiction.
\end{proof}

\subsection*{Acknowledgements}
The authors are grateful to Professor Masahiro Yamamoto and members of his group including Ms. Anna Suzuki for valuable information on this topic.
The authors would like to thank Professor Adam Kubica for careful reading of the manuscript and beneficial discussions.
The authors are grateful to the anonymous referee for improving presentation of the paper.
The first author is partly supported by the Japan Society for the Promotion of Science (JSPS) through grant No. 26220702 (Kiban S) and No. 16H03948 (Kiban B).
The second author is supported by Grant-in-aid for Scientific Research of JSPS Fellows No. 16J03422 and the Program for Leading Graduate Schools, MEXT, Japan.


\end{document}